\documentclass[11pt]{article}

\usepackage[a4paper,margin=2.7cm]{geometry}
\usepackage[T1]{fontenc}
\usepackage[utf8]{inputenc}

\usepackage{amsmath,amssymb,amsfonts,amsthm}
\usepackage{mathtools}
\usepackage{mathrsfs}
\usepackage{bm}
\usepackage{bbm}

\usepackage{graphicx}
\usepackage{xcolor}
\usepackage{hyperref}

\theoremstyle{definition}
\newtheorem{defi}{Definition}[section]

\theoremstyle{plain}
\newtheorem{prop}[defi]{Proposition}
\newtheorem{lem}[defi]{Lemma}
\newtheorem{theorem}[defi]{Theorem}
\newtheorem{coro}[defi]{Corollary}

\theoremstyle{remark}
\newtheorem{rem}[defi]{Remark}

\newcommand{\cL}{\mathcal{L}}
\newcommand{\cT}{\mathcal{T}}
\newcommand{\RR}{\mathbb{R}}
\newcommand{\CC}{\mathbb{C}}
\newcommand{\HH}{\mathbb{H}}
\newcommand{\OO}{\mathbb{O}}
\newcommand{\ZZ}{\mathbb{Z}}
\newcommand{\bu}{\mathbbm{1}}
\newcommand{\buTwo}{\mathbbm{1}_2}
\newcommand{\Spin}{\mathrm{Spin}}
\newcommand{\SO}{\mathrm{SO}}
\newcommand{\Herm}{\mathrm{Herm}}
\newcommand{\ii}{\mathrm{i}}
\newcommand{\st}{^{\!*}} 
\newcommand{\Lm}[1]{L_{#1}} 
\newcommand{\Bf}{B} 

\newcommand{\Rp}{\mathrm{R}} 
\newcommand{\Lp}{\mathrm{L}} 

\newcommand{\DA}{\mathbb{A}}
\DeclareMathOperator{\End}{End}

\title{A $2\times2$ Matrix Approach to Hopf Fibrations:\\
Quaternionic Reduction of Octonionic Bi-unitary Matrices\\ and an RLLR/LRRL Bracketing Rule}

\author{Jean-Pierre Gazeau${}^\ast{}^{\dagger}$ and Aneta Sli\.{z}ewska${}^\ast$\\
${}^\ast$ Faculty of Mathematics, University of Bia\l ystok, 15-245 Bia\l ystok, Poland\\
${}^{\dagger}$Universit\'e Paris Cit\'e, CNRS, Astroparticule et Cosmologie, F-75013 Paris, France\\
\small\textit{j.gazeau@uwb.edu.pl, gazeau@apc.in2p3.fr,  anetasl@uwb.edu.pl}}
\date{\today}

\begin{document}
\maketitle

\begin{abstract}
We develop a common two-by-two matrix framework for the four classical Hopf fibrations
associated with the real, complex, quaternionic and octonionic numbers. In all four cases the
Hopf map is the image of the north pole of a sphere of trace-free Hermitian matrices under
conjugation by a bi-unitary matrix, and it depends only on the first column of that matrix,
without any bracketing convention, even for the non-associative octonions. Our main results
concern general octonionic bi-unitary matrices, beyond the compatible matrices, with entries in
a single complex subalgebra, on which earlier work on octonionic Lorentz transformations relies.
We prove that every octonionic bi-unitary matrix is in fact quaternionic: its entries lie in a
quaternion subalgebra depending on the matrix. We show that requiring Hermitian output leaves
exactly four inequivalent ways of bracketing the octonionic triple products, among them an
interlaced rule called RLLR/LRRL, and that each of them splits into an ordinary
five-dimensional rotation and a four-dimensional scaling governed by a non-commutative
determinant. This determines exactly when the evaluation preserves the geometry of the sphere,
shows that it may contract or even collapse four dimensions, and explains why it cannot be
compatible with matrix multiplication. We also recover in compact form, through the Clifford
map, the known fact that suitable compatible matrices, which we interpret as transvections of
the eight-sphere, generate the spin group $\Spin(9)$ by nested evaluations. Non-associativity
thus does not prevent a two-by-two octonionic description of the full $\Spin(9)$ symmetry; it
prevents that description from being realised through unrestricted multiplication of
octonionic matrices.
\end{abstract}

\noindent\textbf{Keywords:} Hopf fibrations; octonions; non-associativity; bi-unitary matrices;
quaternionic reduction; non-commutative determinants; $\Spin(9)$; Clifford map; bracketing rule.

\tableofcontents

\section{Introduction}
\label{sec:intro}

The four normed division algebras over $\RR$ --- real numbers $\RR$,
complex numbers $\CC$, quaternions $\HH$, and octonions $\OO$ ---
each give rise to a classical Hopf fibration:
\begin{equation}
\label{hopfmap}
\mathbb S^0\hookrightarrow\mathbb S^1\stackrel{\pi}{\rightarrow} \mathbb S^1,\quad \mathbb{S}^1\hookrightarrow \mathbb{S}^3  \stackrel{\pi}{\rightarrow}  \mathbb{S}^2\, , \quad  \mathbb{S}^3\hookrightarrow\mathbb{S}^7 \stackrel{\pi}{\rightarrow}  \mathbb{S}^4\, , \quad  \mathbb{S}^7\hookrightarrow\mathbb{S}^{15}  \stackrel{\pi}{\rightarrow}  \mathbb{S}^8\, ,
\end{equation}
where ``$\hookrightarrow$'' stands for embedding of the fiber and~$\pi$ is the Hopf map.

The Hopf construction, in~its broader application, yields circle bundles
$\mathbb{S}^{2n+1} \stackrel{\pi}{\rightarrow} \CC \mathbf{P}^n$ over complex projective space,
obtained by restricting the tautological line bundle over $\CC \mathbf{P}^n$ to the unit
sphere in $\CC^{n+1}$. In the same way, viewing $\mathbb{S}^{4n-1}\subset\HH^n$ and
factoring out right multiplication by the group of unit quaternions ($\simeq\mathbb S^3$) yields
$\mathbb{S}^{4n-1}\to\HH\mathbf{P}^{n-1}$; for $n=2$, since $\HH\mathbf P^1\cong\mathbb S^4$,
this is the bundle $\mathbb{S}^7 \stackrel{\pi}{\rightarrow} \mathbb{S}^4$ with fiber $\mathbb{S}^3$.
Over the octonions this construction breaks down in general: the unit octonions
$\mathbb S^7$ do not form a group, and identifying $(u_1,\dots,u_n)$ with
$(u_1\lambda,\dots,u_n\lambda)$, $|\lambda|=1$, does not define an equivalence relation on
$\OO^n$ when $n\ge3$. For $n=2$ it can be saved, because any two octonions generate an
associative subalgebra: the octonionic line through $(u,v)\neq0$ is well defined, being
determined by the ratio $uv^{-1}$ (or by $u=0$), and the space of such lines is
$\OO\mathbf P^1\cong\mathbb S^8$. This gives the bundle
$\mathbb{S}^{15} \stackrel{\pi}{\rightarrow} \mathbb{S}^8$ with fiber $\mathbb{S}^7$. (The Cayley
plane $\OO\mathbf P^2$ exists, but it is not the base of any analogous bundle with total space
$\mathbb S^{23}$.) At the bottom of the ladder, $\mathbb{S}^1\subset\RR^2$ modulo multiplication
by $\pm1$ is $\RR\mathbf{P}^1\cong\mathbb{S}^1$, which gives
$\mathbb{S}^1 \stackrel{\pi}{\rightarrow} \mathbb{S}^1$ with fiber $\mathbb{S}^0=\{-1,1\}$. Likewise, a ``circle of latitude'' $\{z=\cos\alpha\}$, $0<\alpha<\pi$, of $\mathbb S^1$ is the
$0$-sphere $\{(\cos\alpha,\pm\sin\alpha)\}$, and its Hopf inverse is the four-point set
$\theta\in\{\pm\frac{\alpha}{2},\pm\frac{\alpha}{2}+\pi\}\cong\mathbb S^0\times\mathbb S^0$,
the real counterpart of the torus $\mathbb S^1\times\mathbb S^1$ found below for $\mathbb S^3\to\mathbb S^2$; for the equator
$z=0$ this is $\theta\in\frac{\pi}{4}+\frac{\pi}{2}\ZZ$ (Lemma~\ref{lem:fibers} with $n=1$).

These four cases are distinguished by a classical rigidity result. If
$\mathbb S^k\hookrightarrow\mathbb S^m\to\mathbb S^n$ is a fiber bundle with $n\ge1$ whose
fiber, total space and base are spheres, then $(k,m,n)\in\{(0,1,1),(1,3,2),(3,7,4),(7,15,8)\}$:
a classical argument based on the long exact homotopy sequence shows that $m=2n-1$,
$k=n-1$ and that the bundle projection has Hopf invariant one, and Adams' theorem on the non-existence of elements of Hopf invariant
one leaves only $n\in\{1,2,4,8\}$ \cite{Hopf,Adams,Husemoller}. (The statement
concerns the dimensions; it does not assert that the bundle is unique in each dimension.)
Among the four, the octonionic case $\mathbb{S}^7\hookrightarrow\mathbb{S}^{15}  \stackrel{\pi}{\rightarrow}  \mathbb{S}^8$
stands apart: the loss of associativity in $\OO$ complicates even the most basic explicit computations.

\bigskip
\noindent\textbf{Motivation from mathematical physics.}
Transformations of the form $X\mapsto gXg^\ddagger$, or $(gX)g^\ddagger$ in the
non-associative case, acting on $2\times2$ matrices $X$ that are Hermitian for a suitable
involution $\ddagger$ and preserving a (pseudo-)metric $B(X,X)$ read off from $X^2$ or
$-\det X$, are ubiquitous in physics. The prototype is the action of $\mathrm{SL}(2,\CC)$
on $X=x^\mu\sigma_\mu$, which realises the Lorentz group as $\mathrm{SL}(2,\CC)/\{\pm1\}$ and
preserves $\det X=x_\mu x^\mu$ \cite{PenroseRindler}. Running through the division algebras
gives $\mathfrak{sl}(2,\DA)\cong\mathfrak{so}(\dim_\RR\DA+1,1)$ for $\DA=\RR,\CC,\HH,\OO$
\cite{Sudbery1984,KugoTownsend1983}, and the resulting dimensions $3,4,6,10$ are precisely
those in which the classical superstring and minimal super-Yang--Mills theories exist
\cite{KugoTownsend1983,Evans1988,BaezHuerta2010}. Changing the involution yields other
physically relevant groups. Examples include the de Sitter group
$\Spin(4,1)\cong\mathrm{Sp}(1,1)$, acting on quaternionic matrices \cite{Takahashi1963}; the
Euclidean conformal group $\Spin(5,1)\cong\mathrm{SL}(2,\HH)$, acting by Möbius maps on
$\HH\cup\{\infty\}\cong\mathbb S^4$ \cite{Ahlfors1985}; and the anti de Sitter and conformal
groups $\SO(3,2)$, $\SO(4,2)$, treated in this language in \cite{gazpejtod26}. In atomic
physics, the dynamical symmetries $\SO(4)\subset\SO(4,1)\subset\SO(4,2)$ of the hydrogen atom
\cite{Fock1935,BarutKleinert1967} admit the same quaternionic description, which yields exact
closed-form multiphoton matrix elements \cite{Gazeau1978,Gazeau1979}. The pattern also extends to internal degrees of freedom. Take for $X$ the ``north pole''
$\mathrm{diag}(1,-1)$, and let $\psi\in\DA^2$ be the first column of a matrix $g$ with
$gg^\dagger=g^\dagger g=\buTwo$, so that $\psi^\dagger\psi=1$. Then
$g\,\mathrm{diag}(1,-1)\,g^\dagger=2\psi\psi^\dagger-\buTwo$, where $\psi\psi^\dagger$
is the projector onto the unit vector $\psi$, i.e.\ the density matrix of a pure state.
The Hopf map, which we recover in this form in Proposition~\ref{prop:hopfmatrix}, is thus
the passage from a pure state to its point on the ``Bloch'' sphere.
For $\DA=\CC$ this is the qubit, while the quaternionic and octonionic
fibrations encode two- and three-qubit entanglement \cite{MosseriDandoloff2001,BernevigChen2003}.
Octonionic and Jordan-algebraic structures entered the description of observables and of
colour early on \cite{JordanvNW1934,GunaydinGursey1973}. Across these examples, the
$2\times2$ formalism $X\mapsto gXg^\ddagger$ appears as a common language for objects of
external (space-time, possibly higher-dimensional or with several times) and internal
nature, which suggests that it reflects a structural rather than accidental feature of
physical symmetries. The present paper tests how far this language survives
non-associativity. Over $\OO$ it survives at the level of the spin group generated by
composed evaluations, but not at the level of the matrix product: outside the compatible
class, preservation of $B(X,X)$ is governed by a non-commutative determinant.

\bigskip
\noindent\textbf{Aim of the paper.}
We use $2\times2$ matrix calculus systematically to describe the four classical Hopf
fibrations, and determine precisely how far that calculus carries over to the octonions.
In all four cases the Hopf map is the image of the north pole of the sphere of trace-free
Hermitian matrices under $X\mapsto (gX)g^\dagger$, evaluated on the first column of a
bi-unitary matrix $g\in\cL_\DA$, without any bracketing convention, also over $\OO$.
Earlier work on octonionic Lorentz transformations
\cite{ManogueSchray1993,ManogueDray1999,DrayManogue2015} restricts to \emph{compatible}
matrices, whose entries lie in a single complex subalgebra of $\OO$, so that the bracketing
problem does not arise, and realises the full symmetry group by nesting such
transformations. Our main results concern \emph{general} octonionic bi-unitary matrices;
they are summarised below, after two elementary reminders.

\bigskip
\noindent\textbf{A reminder of the second Hopf map $\mathbb{S}^3 \stackrel{\pi}{\rightarrow} \mathbb{S}^2$.}
We use the correspondence between elements $ \xi= (\xi_0, \pmb{\xi})=\left(\cos \frac{\theta}{2}, \sin\frac{\theta}{2}\mathbf{n}\right) \in \mathbb{S}^3$, $0\leq \theta \leq 2\pi$, and rotations $\mathcal{R}(\mathbf{ n}, \theta)$ in space \cite{gazeau24}:
\begin{equation*}
\label{S3rotA}
\mathbf{ r}'= \mathcal{R}(\mathbf{ n}, \theta)\bullet \mathbf{r}  \ \Leftrightarrow (0, \mathbf{r}')= \xi (0,\mathbf{r})\bar \xi\,,
\end{equation*}
where we use the quaternionic multiplication rule in scalar-vector notation
\begin{equation*}
qq^{\prime}= (q_0,\mathbf{q})(q^{\prime}_0,\mathbf{q}^{\prime})=(q_0q^{\prime}_0 - \mathbf{q}\cdot\mathbf{q}^{\prime}, q_0\mathbf{q}^{\prime} + q^{\prime}_0 \mathbf{q}+ \mathbf{q} \times \mathbf{q}^{\prime})\,,
\end{equation*}
and $\bar q = (q_0,-\mathbf{q})$ is the quaternionic conjugate of $q$, which is also its inverse if $q=\xi\in \mathbb{S}^3$.
We use the complex $2\times 2$ matrix representation of quaternions
 \begin{equation}
 \label{Mquart}
 \HH \ni q \mapsto \mathsf{M}(q)=  \begin{pmatrix}
 q_0 + \ii q_3     &  -q_2 + \ii q_1  \\
 q_2 + \ii q_1     &  q_0 - \ii q_3
\end{pmatrix},
 \end{equation}
which is an injective homomorphism of real algebras, with $\det \mathsf M(q)=|q|^2$.
It restricts to an isomorphism $\mathbb S^3\cong\mathrm{SU}(2)$, and, since every
nonzero quaternion has a unique polar decomposition $q=r\xi$ with $r>0$ and
$\xi\in\mathbb S^3$, to an isomorphism $\HH^\times\cong\RR_{>0}\times\mathrm{SU}(2)$.
Consider now the element $\xi = (\xi_0, \xi_1, \xi_2, \xi_3)$ of $\mathbb{S}^3\cong \mathrm{SU}(2)$ which brings the north pole of $\mathbb{S}^2$ to the direction $\hat{\mathbf{r}}= (x, y, z) \in \mathbb{S}^2$:
\begin{equation}
\label{poltor}
\xi(0,0,0,1)\bar\xi= (0,x,y,z)\,, \quad x^2 +y^2+ z^2=1.
\end{equation}
Using instead the Hermitian matrix notation $\begin{pmatrix}
  z    &  x +\ii y  \\
  x - \ii y    &  -z
\end{pmatrix}\equiv -\ii \mathsf{M}(0,x,y,z)$, \eqref{poltor} reads as the (co-)adjoint action of $\mathrm{SU}(2)$
on the ``north pole'' $X_N = \begin{pmatrix}
1     &  0  \\
 0     &  -1
\end{pmatrix}$ in the orbit $\mathrm{SU}(2)/\mathrm{U}(1)\sim \mathbb{S}^2$:
 \begin{equation}
\label{matpoltor}
\mathsf{M}(\xi)X_N
\mathsf{M}^{\dag}(\xi)
=\begin{pmatrix}
 z     &  x + \ii y  \\
 x - \ii y     &  -z
\end{pmatrix}, \qquad \mathsf{M}(\xi)\in \mathcal{L}_{\HH}.
\end{equation}
This gives the quadratic relations which exemplify the projection $\pi$ (Hopf's map) of $\mathbb{S}^3$ onto the base space $\mathbb{S}^2$:
\begin{equation}
\label{sechopf}
\mathbb{S}^3 \ni \xi \stackrel{\pi}{\rightarrow}
\left\lbrace
 \begin{array}{cc}
x =& 2(\xi_0 \xi_2 + \xi_1\xi_3)\\
y =& 2(\xi_2 \xi_3 - \xi_0 \xi_1)\\
z =& \xi_0^2 + \xi_3^2 - \xi_1^2 - \xi_2^2
\end{array} \right.\, \in \mathbb{S}^2.
\end{equation}
One easily checks that the Hopf inverse of a point of $\mathbb{S}^2$ is a circle
$\mathbb{S}^1 \subset \mathbb{S}^3$ (the diagonal subgroup $\mathrm U(1)$ fixes $X_N$),
and that the Hopf inverse of a circle of latitude in $\mathbb{S}^2$ is a torus
$\mathbb{S}^1 \times \mathbb{S}^1$: fixing $z=\cos\alpha$ in \eqref{sechopf} forces
$\xi_0^2+\xi_3^2=\cos^2\frac{\alpha}{2}$ and $\xi_1^2+\xi_2^2=\sin^2\frac{\alpha}{2}$, so
$(\xi_0,\xi_3)$ and $(\xi_1,\xi_2)$ independently parametrize the two factor circles of the torus.

For $\HH$ the adjoint action $X\mapsto gXg^\dagger$ on trace-free Hermitian $2\times 2$
matrices is completely unambiguous. Extending this framework to the octonions $\OO$
encounters a key obstacle, non-associativity: the expression $(gX)g^\dagger$ involves triple
products for which the two bracketings $(g_{ik}X_{kj})g^\dagger_{jl}$ and
$g_{ik}(X_{kj}g^\dagger_{jl})$ may differ.

\bigskip
\noindent\textbf{The often overlooked first Hopf map
$\mathbb S^0=\{-1,1\}\hookrightarrow\mathbb S^1
\stackrel{\pi}{\longrightarrow}\mathbb S^1$.}
Topologically, this is the quotient map $\mathbb S^1\to\RR\mathbf P^1$ by the antipodal
involution, followed by the standard identification $\RR\mathbf P^1\cong\mathbb S^1$; in
angular terms it is $e^{\ii\theta}\mapsto e^{2\ii\theta}$. It is a connected two-sheeted
covering of degree two, and therefore not a trivial covering (a trivial two-sheeted
covering of $\mathbb S^1$ would have a disconnected total space). This is the real
counterpart of the non-triviality of the other three Hopf bundles.

It is instructive to redo the matrix construction one step down the Cayley--Dickson
ladder, for $\DA=\RR$, where $\CC$ (represented by real $2\times2$ matrices) plays the role
that $\HH$ (represented by complex $2\times2$ matrices) just played above.
Represent $\xi=(\xi_0,\xi_1)\in\mathbb S^1\subset\RR^2\cong\CC$ by the real rotation matrix
\begin{equation}
\label{MxiR}
\RR^2\cong\CC\ \ni\ \xi\ \longmapsto\ \mathsf M(\xi)=
\begin{pmatrix}\xi_0 & -\xi_1\\ \xi_1 & \xi_0\end{pmatrix}\in \mathrm{SO}(2),
\qquad \xi_0^2+\xi_1^2=1,
\end{equation}
and identify the trace-free symmetric real $2\times2$ matrices
$X(z,x)=\begin{pmatrix}z&x\\ x&-z\end{pmatrix}$ with $z^2+x^2=1$ with the unit circle of
$\Herm_2(\RR)_0\cong\RR^2$ (the $\DA=\RR$ row of Table~\ref{tab:cases}).
Take the north pole $X_N=\mathrm{diag}(1,-1)$. A direct computation gives
\begin{equation}
\label{PoltorR}
 \mathsf M(\xi) X_N {}^\mathrm{t}\mathsf M(\xi)
=\begin{pmatrix} \xi_0^2-\xi_1^2 & 2\xi_0\xi_1\\ 2\xi_0\xi_1 & -(\xi_0^2-\xi_1^2)\end{pmatrix},\qquad \mathsf M(\xi)\in \mathcal{L}_{\CC},
\end{equation}
so that, as in \eqref{sechopf},
\begin{equation}
\label{firsthopf}
\mathbb{S}^1 \ni \xi \stackrel{\pi}{\rightarrow}
\left\lbrace
 \begin{array}{cc}
x =& 2\xi_0 \xi_1\\
z =& \xi_0^2  - \xi_1^2
\end{array} \right.\, \in \mathbb{S}^1.
\end{equation}
With $\xi_0=\cos\theta$, $\xi_1=\sin\theta$ this is $(z,x)=(\cos2\theta,\sin2\theta)$:
the matrix construction reproduces precisely the angle-doubling map
$\theta\mapsto 2\theta \pmod{2\pi}$, with fiber
$\pi^{-1}(z,x)=\{\theta,\theta+\pi\}\cong\mathbb S^0$. This is a fully explicit
$2\times2$ matrix illustration of Proposition~\ref{prop:hopfmatrix} at the bottom of the
Cayley--Dickson tower, before the complications of $\HH$ and $\OO$ appear.

\bigskip
\noindent\textbf{Three distinct achievements.}
For the octonionic case it is essential to distinguish between three statements of
increasing strength, which are sometimes conflated:
\begin{enumerate}
\item[(H)] \emph{recovering the Hopf map} $\pi:\mathbb S^{15}\to\mathbb S^8$ as the
image of the north pole under a matrix formula;
\item[(T)] obtaining, for each bi-unitary matrix $g\in \mathcal{L}_{\OO}$, an \emph{algebraic transformation}
$X\mapsto\mathrm{Eval}(g;X)$ of the target $\Herm_2(\OO)_0\supset\mathbb S^8$;
\item[(S)] \emph{realising the $\Spin(9)$ symmetry}: exhibiting a spinorial action $S$ on
$\OO^2\cong\RR^{16}$ and an orthogonal action $\rho(S)$ on $\RR^9$ such that
$\pi(S\xi)=\rho(S)\pi(\xi)$, and identifying $\mathrm{Eval}$ with $\rho$.
\end{enumerate}
Statement (H) holds uniformly and without any bracketing convention
(Proposition~\ref{prop:hopfmatrix}). Statement (T) requires a convention; the RLLR/LRRL rule
provides one that preserves Hermiticity (Section~\ref{sec:mnemonic}). Statement (S) is
classical at the level of $\Spin(9)\subset\SO(16)$ (Section~\ref{sec:Spin9}), and it is also
known that it can be realised inside octonionic $2\times2$ matrix calculus by nesting compatible
transformations \cite{ManogueSchray1993,ManogueDray1999,DrayManogue2015} (Section~\ref{sec:transvections}). What has not been
analysed, to our knowledge, is what (T) does for a \emph{general} bi-unitary octonionic matrix;
this is the subject of our main results (Section~\ref{sec:blocks}).

\bigskip
\noindent\textbf{Main results: quaternionic reduction and a $5+4$ splitting.}
What happens for a general $g\in\cL_\OO$ is described exactly by Lemma~\ref{lem:quatred} and
Theorem~\ref{thm:blocks}.
\begin{enumerate}
\item[(i)] \emph{Quaternionic reduction.} The bi-unitarity constraints force the four entries of
$g$ into a common quaternion subalgebra $\HH'\subset\OO$, so that
$\cL_\OO=\bigcup_{\HH'}\mathrm{Sp}(2)_{\HH'}$. Compatible matrices are the special case in
which the entries already lie in a complex subalgebra.
\item[(ii)] \emph{Classification of bracketings.} Requiring trace-free Hermitian output for all
$g$ leaves exactly four inequivalent bracketing rules (Proposition~\ref{prop:admissible}).
\item[(iii)] \emph{$5+4$ splitting.} Decomposing $\OO=\HH'\oplus\ell'\HH'$, every admissible
evaluation acts as the ordinary quaternionic $\Spin(5)$ rotation on $\RR\oplus\HH'\cong\RR^5$
and as left multiplication by $\overline{\Delta(g)}$ on $\ell'\HH'\cong\RR^4$, where $\Delta(g)$
is one of four non-commutative $2\times2$ determinants; for RLLR/LRRL, $\Delta(g)=ad-cb$ is the
column determinant.
\end{enumerate}
Hence the singular values are $1^{(5)}$ and $|\Delta(g)|^{(4)}$; the evaluation is an isometry if
and only if $|\Delta(g)|=1$, a contraction otherwise, and singular when $\Delta(g)=0$; and the
failure of compatibility with the matrix product is exactly the non-multiplicativity of the
column determinant on $\mathrm{Sp}(2)$. For the special form~\eqref{eq:gspecial} one finds
$|\Delta|^2=1-|ab-ba|^2$, so the defect is measured by the commutator of the first row. The four
rules, which Hermiticity cannot distinguish, are thus distinguished metrically, which gives the
RLLR/LRRL rule an intrinsic characterisation (Remark~\ref{rem:coldet}).

\bigskip
\noindent\textbf{The compatible case: transvections and $\Spin(9)$.}
For the bi-unitary matrices
\[
g_{\omega,\mu}=\begin{pmatrix}\cos\omega & -\sin\omega\,\mu\\ \sin\omega\,\bar\mu & \cos\omega\end{pmatrix},
\qquad \omega\in\RR,\ \mu\in\OO,\ |\mu|=1,
\]
whose entries lie in the complex subalgebra spanned by $1$ and $\mu$, all bracketings of
$(gX)g^\dagger$ coincide, and, writing $X=\begin{psmallmatrix}y&p\\\bar p&-y\end{psmallmatrix}$,
\begin{equation}
\label{eq:intro-rot}
(g_{\omega,\mu}X)g_{\omega,\mu}^\dagger=\begin{pmatrix}y'&p'\\\overline{p'}&-y'\end{pmatrix},\qquad
\begin{aligned}
y'&=y\cos2\omega-\sin2\omega\,\langle p,\mu\rangle,\\
p'&=p+\bigl(y\sin2\omega-2\sin^2\!\omega\,\langle p,\mu\rangle\bigr)\mu,
\end{aligned}
\end{equation}
the rotation by $2\omega$ in the plane spanned by the north pole and $\mu$. That such compatible
matrices act by rotations, and that nested transformations of this kind generate the whole
symmetry group, is known from the work of Manogue, Schray and Dray on octonionic Lorentz
transformations \cite{ManogueSchray1993,ManogueDray1999,DrayManogue2015}; the corresponding infinitesimal statement goes back to
Sudbery \cite{Sudbery1984}. What we add (Theorem~\ref{thm:transvection},
Corollary~\ref{coro:words}) is a short proof in the present bi-unitary setting, based on the
factorisation $T_{g_{\omega,\mu}}=\Phi(Y_{\omega,\mu})\Phi(X_N)$ into two Clifford reflections,
and the interpretation of $\rho(T_{g_{\omega,\mu}})$ as a transvection of the symmetric space
$\mathbb S^8$, whence the name. Together with the main results, this shows that composing
evaluations, rather than multiplying matrices, is unavoidable: for the product $g_1g_2$ of two
transvections, which is again bi-unitary, \emph{no} Hermiticity-preserving bracketing rule gives
an isometry (Proposition~\ref{prop:counter}).

\bigskip
\noindent\textbf{Relation to previous work.}
The $\Spin(9)$ geometry of the octonionic Hopf fibration and its equivariant description
are classical and have been extensively developed using Clifford systems, spin
representations and differential-geometric methods
\cite{Harvey,LawsonMichelsohn,OrneaPartonPiccinniVuletescu,PartonAxioms}; there octonions enter
only through the $16\times16$ operators $\Lm p$, whereas the present paper asks how much of this
symmetry can be expressed within octonionic $2\times2$ matrix calculus itself. The octonions
themselves are treated in detail in \cite{Baez,Schafer,Okubo,SpringerVeldkamp,ConwaySmith}.
The use of $2\times2$ octonionic Hermitian matrices for the vector representation goes back to
Sudbery's identification of $\mathfrak{sl}(2,\OO)$ with $\mathfrak{so}(9,1)$ \cite{Sudbery1984};
the finite transformations, unambiguous for compatible matrices and generating the whole
Lorentz group by nesting, were constructed by Manogue, Schray and Dray
\cite{ManogueSchray1993,ManogueDray1999,DrayManogue2015}. Our treatment of transvections
(Section~\ref{sec:transvections}) is the compact, Euclidean counterpart of that construction,
and we claim no priority for it. The results listed above under (i)--(iii) concern what lies
\emph{outside} the compatible class; to our knowledge they have not appeared in the literature
on octonionic matrices, where non-compatible matrices are avoided rather than analysed. A recent
paper \cite{gazeau26} by one of us develops a related mnemonic matrix rule for (split)
octonionic multiplication and the Cayley--Dickson tower; the relation between that rule and the
present one is left to future work (Remark~\ref{rem:operatorscope}).

\bigskip
\noindent\textbf{Organisation.}
Section~\ref{sec:algebras} fixes the algebraic setting. Section~\ref{sec:matrix} sets up the
common $2\times2$ model (the bi-unitary set $\cL_\DA$, the sphere $\Herm_2(\DA)_0$, the Hopf map
as image of the north pole). Section~\ref{sec:Spin9} gives a self-contained account of the
Clifford map and of $\Spin(9)$, the reference against which the matrix calculus is tested.
Section~\ref{sec:mnemonic} introduces the RLLR/LRRL rule and classifies the
Hermiticity-preserving bracketings. Section~\ref{sec:scope} treats the transvections and proves
the main results, with explicit examples. Section~\ref{sec:general} separates what holds for
general algebras with involution, including the split octonions, and Section~\ref{sec:conclusion}
concludes.

\section{Algebraic preliminaries}
\label{sec:algebras}

We work with two levels of generality: a weak one, sufficient for the Hermitian
structure of the evaluation, and a strong one (Euclidean composition algebras), needed
for metric statements.

\begin{defi}[Algebra with standard involution]
\label{def:invol}
An \emph{algebra with standard involution} is a finite-dimensional unital real algebra
$\DA$ equipped with an involutive anti-automorphism $x\mapsto\bar x$,
\[
\overline{xy}=\bar y\,\bar x,\qquad \overline{\bar x}=x,\qquad \overline{1}=1,
\]
such that $x+\bar x\in\RR1$ and $x\bar x=\bar x x\in\RR1$ for all $x$. We write
$\Re(x):=\tfrac12(x+\bar x)$ and $N(x):=x\bar x$, and we say that $\DA$ is
\emph{trace-associative} if, for all $x,y,z\in\DA$,
\begin{align}
\label{eq:real-assoc}
\Re\bigl((xy)z\bigr)&=\Re\bigl(x(yz)\bigr), \\
\label{eq:real-cyclic}
\Re\bigl((xy)z\bigr)&=\Re\bigl((yz)x\bigr)=\Re\bigl((zx)y\bigr).
\end{align}
\end{defi}

\begin{defi}[Euclidean composition algebra]
\label{def:normed}
A \emph{Euclidean composition algebra} is an algebra with standard involution for which
$N(x)=|x|^2$ is positive definite and multiplicative, $|xy|=|x|\,|y|$. The inner
product is $\langle x,y\rangle:=\Re(x\bar y)=\Re(\bar xy)$.
\end{defi}

\begin{rem}[What these definitions cover]
\label{rem:hurwitz}
By Hurwitz' theorem, the Euclidean composition algebras are, up to isomorphism, exactly
$\RR$, $\CC$, $\HH$ and $\OO$, of dimensions $1,2,4,8$ \cite{Schafer,SpringerVeldkamp}. Thus
Definition~\ref{def:normed} does \emph{not} provide additional generality beyond the four
normed division algebras; it only fixes the axioms we use.
By contrast, Definition~\ref{def:invol} is genuinely more general. It includes all
composition algebras, in particular the split octonions (where $\ell^2=+1$
in~\eqref{eq:CD} below), whose quadratic form $N$ is \emph{nondegenerate but indefinite},
of signature $(4,4)$. 
\end{rem}

The Cayley--Dickson construction provides a uniform description.
Viewing $\OO$ as pairs of quaternions $(q,p)\in\HH\oplus\ell\,\HH$,
the multiplication reads
\begin{equation}
\label{eq:CD}
(q_1+\ell p_1)(q_2+\ell p_2)
= q_1 q_2 + \ell^2\,p_2 p_1\st \;+\;\ell\,(q_2 p_1+q_1\st p_2),
\end{equation}
where $\st$ denotes quaternionic conjugation and $\ell^2=-1$ for the division algebra
(the split octonions correspond to $\ell^2=+1$; see \cite{gazeau26,gazpejtod26}).
Here $\HH$ has its usual basis $1,\ii,\mathrm j,\mathrm k$, with
$\ii^2=\mathrm j^2=\mathrm k^2=\ii\mathrm j\mathrm k=-1$, and $\ell$ denotes the element $(0,1)$,
a unit orthogonal to $\HH$ with $\ell^2=-1$. By~\eqref{eq:CD}, the product of $\ell$ and
$q\in\HH$ is $\ell q=(0,q)$, while $q\ell=\ell\bar q$; thus $\OO$ has the orthonormal basis
$$
1,\ \ii,\ \mathrm j,\ \mathrm k,\ \ell,\ \ell\ii,\ \ell\mathrm j,\ \ell\mathrm k,
$$
and, for instance, $\ell\ii$ is the element $(0,\ii)$.
This is the convention used in all explicit examples below.

\begin{defi}[Associativity type]
\label{def:assoctype}
With the associator $[x,y,z]:=(xy)z-x(yz)$, we say that $\DA$ is
\emph{associative} if $[x,y,z]=0$ identically, and \emph{alternative} if $[x,y,z]$ is an
alternating function of its arguments, equivalently $[x,x,y]=[y,x,x]=0$.
Among the normed division algebras, $\RR,\CC,\HH$ are associative and $\OO$ is
alternative but not associative. Alternativity implies flexibility, $[x,y,x]=0$, and,
by Artin's theorem, the subalgebra generated by any two elements is associative.
\end{defi}

\begin{lem}[Real part identities]
\label{lem:realpart}
Let $\DA$ be an alternative algebra with standard involution (for instance $\OO$, or the
split octonions). Then $\DA$ is trace-associative, i.e.\
\eqref{eq:real-assoc}--\eqref{eq:real-cyclic} hold. Moreover, in any algebra with standard
involution,
\begin{equation}
\label{eq:conj-swap}
\overline{(xy)z}=\bar z\,(\bar y\,\bar x),
\qquad
\overline{x(yz)}=(\bar z\,\bar y)\,\bar x,
\end{equation}
that is, conjugation reverses the order of factors and swaps the bracketing type.
\end{lem}

\begin{proof}
Equation~\eqref{eq:conj-swap} follows from applying $\overline{uv}=\bar v\,\bar u$ twice.
It gives $\overline{[x,y,z]}=\bar z(\bar y\bar x)-(\bar z\bar y)\bar x=-[\bar z,\bar y,\bar x]$.
Since $\bar w=2\Re(w)-w$ and the associator vanishes when one argument is real, alternativity
gives $[\bar z,\bar y,\bar x]=-[z,y,x]=[x,y,z]$, hence $\overline{[x,y,z]}=-[x,y,z]$ and
$\Re[x,y,z]=0$, which is~\eqref{eq:real-assoc}. For~\eqref{eq:real-cyclic}, the identity
$\Re(uv)=\Re(vu)$ (write $u=s+u'$, $v=t+v'$ with $s,t$ real and $u',v'$ imaginary; then
$\Re(uv)=st+\Re(u'v')$, and $\Re(u'v')=\tfrac12(u'v'+v'u')$ is symmetric) gives $\Re((xy)z)=\Re(z(xy))$, which equals $\Re((zx)y)$
by~\eqref{eq:real-assoc}; iterating gives the cyclic invariance. See also
\cite[Ch.~III]{Schafer}, \cite[\S1]{Okubo}, \cite[\S2]{Baez}.
\end{proof}
Since composition algebras are alternative, Lemma~\ref{lem:realpart} applies to all of them. We have also
checked numerically that~\eqref{eq:real-assoc}--\eqref{eq:real-cyclic} hold in the
sixteen-dimensional sedenions, which are neither alternative nor composition algebras;
we do not use this fact.
We shall also use left multiplications $\Lm x:\OO\to\OO$, $\Lm x(y)=xy$, regarded as
real-linear operators on $\OO\cong\RR^8$ with its Euclidean structure.

\begin{lem}[Left multiplications in $\OO$]
\label{lem:Lx}
For all $x,y\in\OO$ and every unit $\mu\in\OO$:
\begin{enumerate}
\item $\Lm x^{\mathrm t}=\Lm{\bar x}$ (adjoint with respect to $\langle\cdot,\cdot\rangle$);
\item $\Lm x\Lm{\bar y}+\Lm y\Lm{\bar x}=2\langle x,y\rangle\,\bu_8$; in particular
$\Lm x\Lm{\bar x}=|x|^2\bu_8$;
\item $\mu\,x\,\mu=2\langle\mu,\bar x\rangle\,\mu-\bar x$, the left-hand side being
unambiguous by flexibility.
\end{enumerate}
\end{lem}

\begin{proof}
(1) $\langle xy,z\rangle=\Re((xy)\bar z)=\Re(y(\bar z x))=\Re(y\,\overline{\bar x z})
=\langle y,\bar xz\rangle$, by~\eqref{eq:real-assoc}--\eqref{eq:real-cyclic}.
(2) By left alternativity, $x(\bar x z)=x((2\Re x-x)z)=2\Re(x)\,xz-(xx)z=(2\Re(x)x-x^2)z=|x|^2z$;
polarising in $x$ gives the claim.
(3) Polarising $u\bar u=|u|^2$ gives $u\bar v+v\bar u=2\langle u,v\rangle$; with $u=x$,
$v=\bar\mu$ this reads $x\mu=2\langle x,\bar\mu\rangle-\bar\mu\bar x$, so
$\mu(x\mu)=2\langle x,\bar\mu\rangle\mu-\mu(\bar\mu\bar x)=2\langle \mu,\bar x\rangle\mu-\bar x$
by (2).
\end{proof}

\section{The $2\times2$ Matrix Framework for Hopf Fibrations}
\label{sec:matrix}

\subsection{Bi-unitary matrices}
\label{sec:biunitary}

\begin{defi}[Bi-unitary set $\cL_\DA$]
\label{def:Ldef}
For an algebra with standard involution $\DA$, define
\begin{equation}
\label{eq:Ldef}
\cL_\DA
=
\left\{
g=\begin{pmatrix}a & b\\ c & d\end{pmatrix}
:\ a,b,c,d\in\DA,\quad gg^\dagger=g^\dagger g=\buTwo
\right\},
\qquad g^\dagger:=\overline{{}^t g}.
\end{equation}
\end{defi}

Unwinding~\eqref{eq:Ldef} yields the constraints
\begin{align}
\label{eq:constraints1}
N(a)+N(b)=1,\quad N(c)+N(d)=1,\quad a\bar c+b\bar d=0= c\bar a+d\bar b, \\
\label{eq:constraints2}
N(a)+N(c)=1,\quad N(b)+N(d)=1,\quad \bar ab+\bar cd=0=\bar ba+\bar dc.
\end{align}

\begin{rem}[Why Definition~\ref{def:Ldef} needs no bracketing convention]
\label{rem:eq5safe}
The matrix products $gg^\dagger$ and $g^\dagger g$ in~\eqref{eq:Ldef} involve only
\emph{binary} products: each entry is a sum of two ordinary products, e.g.\
$(gg^\dagger)_{11}=a\bar a+b\bar b$. Hence $\cL_\DA$ is defined unambiguously for every
$\DA$, associative or not. Likewise, the product $gh$ of two octonionic matrices is itself
unambiguous, since each entry involves only binary octonionic products. Non-associativity
becomes relevant when such binary expressions are \emph{composed}: in triple products such
as $(gX)g^\dagger$ (Section~\ref{sec:mnemonic}), when iterating matrix products, and when comparing
octonionic matrix multiplication with composition of the corresponding real-linear operators
$T_g\circ T_h$ (Section~\ref{sec:scope}).
\end{rem}

\begin{lem}[Equal norms]
\label{lem:equalnorms}
For any $g\in\cL_\DA$ one has $N(a)=N(d)$ and $N(b)=N(c)$. If $\DA$ is a Euclidean
composition algebra there exists $\omega\in[0,\pi/2]$ with
\begin{equation}
\label{eq:omega}
|a|=|d|=\cos\omega,\qquad |b|=|c|=\sin\omega.
\end{equation}
\end{lem}

\begin{proof}
Subtract the relations $N(a)+N(b)=1$ and $N(a)+N(c)=1$, and similarly for $d$.
\end{proof}

\begin{rem}[Status of $\cL_\DA$ as a group]
\label{rem:Lgroup}
In the associative cases $\DA\in\{\RR,\CC,\HH\}$, $\cL_\DA$ is closed under matrix
multiplication and one recovers the classical compact groups
\begin{equation}
\label{eq:Lassoc}
\cL_\RR\simeq \mathrm O(2),\qquad \cL_\CC\simeq \mathrm U(2),\qquad \cL_\HH\simeq \mathrm{Sp}(2)\simeq\mathrm{USp}(4)\simeq\Spin(5),
\end{equation}
the last identification being the classical rank-$2$ coincidence $B_2=C_2$
($\dim\mathrm{Sp}(2)=2(2\cdot2+1)=10=\binom{5}{2}=\dim\Spin(5)$).\footnote{Here $\mathrm{Sp}(2)$
denotes the \emph{compact} group of $2\times2$ quaternionic bi-unitary matrices (real
dimension $10$), i.e.\ the quaternionic-rank convention $\mathrm{Sp}(n)=\{g\in
\mathrm{M}_n(\HH):gg^\dagger=g^\dagger g=\bu_n\}$. In the complex-dimension convention, indexed by
the dimension of the underlying complex symplectic space ($\HH^2\cong\CC^4$), the same
group is written $\mathrm{USp}(4)$. Throughout this paper $\mathrm{Sp}(2)$ always
means the former group, and must not be confused with the \emph{non-compact} groups
$\mathrm{Sp}(2,\RR)\cong \mathrm{SL}(2,\RR)$ or $\mathrm{Sp}(2,\CC)$
(dimensions $3$ and $6$ respectively).}
For $\DA=\OO$, $\cL_\OO$ is \emph{not} closed under the octonionic matrix product. We shall
see (Lemma~\ref{lem:quatred}) that $\cL_\OO$ is the union of the copies
$\mathrm{Sp}(2)_{\HH'}$ of $\mathrm{Sp}(2)$ attached to the quaternion subalgebras
$\HH'\subset\OO$; a product of elements taken from different copies generally leaves $\cL_\OO$.
For instance, with $g_0$ as in~\eqref{eq:g0} below, $g_0\,\mathrm{diag}(1,\mathrm j)=
\begin{psmallmatrix}\frac{1+\ii}2&\frac{\ell\mathrm j}{\sqrt2}\\ \frac{\ell\ii}{\sqrt2}&\frac{\mathrm j-\mathrm k}2\end{psmallmatrix}$
has $a\bar c+b\bar d=-\ell\ii/\sqrt2\neq0$. Even when a product stays in $\cL_\OO$, the induced
transformations need not compose accordingly (Proposition~\ref{prop:counter}(c)).
\end{rem}

\begin{rem}[Why $\RR$, and only $\RR$, disconnects]
\label{rem:connectedness}
The three groups in \eqref{eq:Lassoc} do not behave uniformly with respect to
connectedness. For $\DA=\RR$, conjugation is trivial and the bi-unitary condition reads
$g{}^\mathrm{t}g={}^\mathrm{t}g\,g=\buTwo$, so $\det(g)=\pm1$; both signs occur and
$\mathrm O(2)$ has two components. For $\DA=\CC$, $\det(g)$ ranges over the connected
group $\mathrm U(1)$, and $\mathrm U(2)$ is connected; for $\HH$, the Dieudonn\'e
determinant $\mathrm{Ddet} g=\vert a\vert \vert d-b a^{-1}c\vert$ ($a\neq 0$) takes values in $\RR_{>0}$, so the bi-unitary constraint forces it to be $1$,
and $\mathrm{Sp}(2)$ is connected, indeed simply connected, matching $\Spin(5)$.
In short, $\mathrm O(2)$ is disconnected only because
$\RR^\times\simeq\{\pm1\}\times\RR_{>0}$ is. None of this is new: the non-commutative
determinant and the connectedness of $\mathrm{Sp}(n)$ go back to \cite{Dieudonne1943}, with
the classical groups and their topology treated in \cite{Dieudonne1955} and, from the
Lie-theoretic viewpoint (including $\mathrm{Sp}(2)\simeq\Spin(5)$), in
\cite{Helgason1978}; see also \cite{Aslaksen1996}.
\end{rem}

\subsection{The Hermitian sphere model}
\label{sec:Herm2}

\begin{defi}[Hermitian matrices and sphere model]
\label{def:Herm2}
For an algebra with standard involution $\DA$, define
$$
\Herm_2(\DA)
=
\left\{\begin{pmatrix} r&x\\\bar x&s\end{pmatrix}:\ r, s\in\RR,\ x\in\DA\right\}
\cong\RR^{2+\dim_\RR\DA},
$$
and its trace-free slice
$$
\Herm_2(\DA)_0
=
\left\{X(y,p):=\begin{pmatrix}y&p\\\bar p&-y\end{pmatrix}:\ y\in\RR,\ p\in\DA\right\}
\cong\RR^{1+\dim_\RR\DA}.
$$
One has $X^2= (y^2+N(p))\buTwo$. We set $Q(X(y,p))=y^2+N(p)$, with polar form
\[\Bf\bigl(X(y,p),X(y',p')\bigr)=yy'+\langle p,p'\rangle.\] If $\DA$ is a Euclidean
composition algebra, $Q$ is positive definite and
$$
\mathbb S^{\dim_\RR\DA}
=\{(y,p)\in\RR\oplus\DA:\ y^2+|p|^2=1\}
$$
is the unit sphere of $\Herm_2(\DA)_0$. The north pole is $X_N=X(1,0)$.
\end{defi}

\begin{table}[h]
\centering
\small\renewcommand{\arraystretch}{1.3}
\begin{tabular}{ccccc}
\hline
$\DA$ & $\dim_\RR\DA$ & sphere & $\cL_\DA$ & group acting on the sphere \\
\hline
$\RR$ & $1$ & $\mathbb S^1$ & $\mathrm{O}(2)$ (group) & $\mathrm O(2)/\{\pm\buTwo\}\cong\mathrm O(2)$ \\
$\CC$ & $2$ & $\mathbb S^2$ & $\mathrm{U}(2)$ (group) & $\SO(3)\cong\mathrm{U}(2)/\mathrm U(1)$ \\
$\HH$ & $4$ & $\mathbb S^4$ & $\mathrm{Sp}(2)\simeq\Spin(5)$ (group) & $\SO(5)\cong\Spin(5)/\{\pm1\}$ \\
$\OO$ & $8$ & $\mathbb S^8$ & $\bigcup_{\HH'}\mathrm{Sp}(2)_{\HH'}$ (set) & $\SO(9)\cong\Spin(9)/\{\pm1\}$, Section~\ref{sec:Spin9} \\
\hline
\end{tabular}
\caption{The four normed division algebra cases in the $2\times2$ matrix framework. In the
first three rows the action is $X\mapsto gXg^\dagger$; in the octonionic row $\Spin(9)$
is \emph{not} $\cL_\OO$, but is generated by the operators $T_g$ attached to the
transvections $g\in\cT\subset\cL_\OO$ (Theorem~\ref{thm:transvection}).}
\label{tab:cases}
\end{table}

\subsection{The Hopf map as the image of the north pole}
\label{sec:hopfmatrix}

At the north pole all bracketings coincide, since $X_N$ has real entries. This makes
the following statement valid uniformly, without any convention.

\begin{prop}[Matrix model of the Hopf map]
\label{prop:hopfmatrix}
Let $\DA\in\{\RR,\CC,\HH,\OO\}$ and define
\begin{equation}
\label{eq:pi}
\pi:\ \mathbb S^{2\dim_\RR\DA-1}\subset\DA^2\to\mathbb S^{\dim_\RR\DA},\qquad
\pi(u,v)=\bigl(|u|^2-|v|^2,\;2u\bar v\bigr).
\end{equation}
For every $g=\begin{psmallmatrix}a&b\\c&d\end{psmallmatrix}\in\cL_\DA$, the matrix $(gX_N)g^\dagger$
is unambiguous, belongs to $\Herm_2(\DA)_0$, and
\begin{equation}
\label{eq:hopfcol}
(gX_N)g^\dagger = X\bigl(\pi(a,c)\bigr),
\end{equation}
where $(a,c)=g\binom10\in\mathbb S^{2\dim_\RR\DA-1}$ is the first column of $g$. The
map $\pi$ is the Hopf fibration: it is surjective and each fiber is a sphere
$\mathbb S^{\dim_\RR\DA-1}$.
\end{prop}

\begin{proof}
Since $X_N=\mathrm{diag}(1,-1)$ has real entries, every entry of $(gX_N)g^\dagger$ is a
sum of binary products: $(gX_N)g^\dagger=\begin{psmallmatrix}|a|^2-|b|^2 & a\bar c-b\bar d\\
c\bar a-d\bar b & |c|^2-|d|^2\end{psmallmatrix}$. By Lemma~\ref{lem:equalnorms},
$|a|^2-|b|^2=|a|^2-|c|^2=-(|c|^2-|d|^2)$, and by~\eqref{eq:constraints1}, $a\bar c-b\bar d=2a\bar c$
and $c\bar a-d\bar b=2c\bar a=\overline{2a\bar c}$. This is~\eqref{eq:hopfcol}, and
$|a|^2+|c|^2=1$ by~\eqref{eq:constraints2}.
For the fibers: $|\pi(u,v)|^2=(|u|^2-|v|^2)^2+4|u|^2|v|^2=1$. Given $(y,p)\in\mathbb S^{\dim\DA}$
with $y\neq-1$, the equations $|u|^2=\frac{1+y}2>0$, $u\bar v=\frac p2$ are solved by
$u$ on the sphere of radius $\sqrt{(1+y)/2}$ and $v=\overline{u^{-1}(p/2)}
=\overline{\bar u\,p}/(1+y)=\bar p\,u/(1+y)$, uniquely determined by $u$ (alternativity is
used in $u^{-1}(u\bar v)=\bar v$). Thus the fiber is a sphere of dimension $\dim_\RR\DA-1$;
over $y=-1$ it is $\{(0,v):|v|=1\}$. Surjectivity is clear.
\end{proof}

Every point of $\mathbb S^{2\dim\DA-1}$ is the first column of some $g\in\cL_\DA$
(Theorem~\ref{thm:transvection} gives an explicit choice in the octonionic case;
the special form below gives another), so~\eqref{eq:hopfcol} realises the whole fibration.

\paragraph{The special form.} A convenient family is
\begin{equation}
\label{eq:gspecial}
g=\begin{pmatrix}a & b\\ c & \bar a\end{pmatrix},\qquad
c=-\bar\alpha\,\bar b\,\alpha,\quad \alpha=a/|a|\qquad (a\neq0),
\end{equation}
where $\bar\alpha\bar b\alpha$ is unambiguous by flexibility. One checks directly from
alternativity that~\eqref{eq:gspecial} belongs to $\cL_\DA$ for every $(a,b)\in\mathbb S^{2\dim\DA-1}$
with $a\neq0$ (for $\DA=\HH$ this is the matrix $\begin{psmallmatrix}a&b\\-\overline{b_\alpha}&\bar a\end{psmallmatrix}$,
$b_\alpha=\bar a b a/|a|^2$). For it, $a\bar c=-b a$, so~\eqref{eq:hopfcol} reads
\begin{equation}
\label{eq:hopfgeneral}
y=|a|^2-|b|^2,\qquad p=-2ba .
\end{equation}
\emph{The case $a=0$.} Then $|b|=|c|=1$ and $d=0$, the constraints hold for every unit $c$,
and~\eqref{eq:hopfcol} gives the south pole $X(-1,0)$, which is also the continuous
extension of~\eqref{eq:hopfgeneral} to $a=0$. Thus the map
$(a,b)\mapsto(|a|^2-|b|^2,-2ba)$ is defined on all of $\mathbb S^{2\dim\DA-1}$; it equals
$\kappa\circ\pi\circ\sigma$ with the isometries $\sigma(a,b)=(\bar a,-b)$ of
$\mathbb S^{2\dim\DA-1}$ and $\kappa(y,p)=(y,\bar p)$ of $\mathbb S^{\dim\DA}$, since
$2\bar a\,\overline{(-b)}=-2\bar a\bar b=\overline{-2ba}$. It is therefore a Hopf fibration,
equivalent to~\eqref{eq:pi}.

\paragraph{(i) Quaternionic case ($\mathbb S^7\to\mathbb S^4$).}
Here $\cL_\HH\simeq\mathrm{Sp}(2)$, the adjoint action $X\mapsto gXg^\dagger$ on
$\Herm_2(\HH)_0\cong\RR^5$ is unambiguous and is the standard action of
$\Spin(5)\simeq\mathrm{Sp}(2)$ on $\mathbb S^4$ through $\SO(5)$. The fiber over each point
of $\mathbb S^4$ is $\mathbb S^3$, while the Hopf inverse of an equatorial
$\mathbb S^3\subset\mathbb S^4$ is $\mathbb S^3\times\mathbb S^3\cong\Spin(4)$
(Lemma~\ref{lem:fibers}). See also Remark \ref{rem:fibersgroups}.

\paragraph{(ii) Octonionic case ($\mathbb S^{15}\to\mathbb S^8$).}
Proposition~\ref{prop:hopfmatrix} applies verbatim, with $\pi(u,v)=(|u|^2-|v|^2,2u\bar v)$
and $y^2+|p|^2=(|u|^2+|v|^2)^2=1$ on $\mathbb S^{15}$; in matrix form
\[
\pi(u,v)\longleftrightarrow
X(y,p)=\begin{pmatrix}|u|^2-|v|^2 & 2u\bar v\\ 2v\bar u & -(|u|^2-|v|^2)\end{pmatrix}
\in\Herm_2(\OO)_0 .
\]
This is statement (H) of the Introduction. It uses nothing beyond binary products, and in
particular it does not depend on the bracketing rule of Section~\ref{sec:mnemonic}.

\begin{lem}[Inverse images of equators]
\label{lem:fibers}
Let $\pi:\mathbb S^{2n-1}\to\mathbb S^n$ be the Hopf fibration for $\DA$ of dimension
$n\in\{1,2,4,8\}$, and let $\Sigma\cong\mathbb S^{n-1}\subset\mathbb S^n$ be an equatorial
sphere. Then $\pi^{-1}(\Sigma)\cong\mathbb S^{n-1}\times\mathbb S^{n-1}$, that is, four points,
a torus $\mathbb S^1\times\mathbb S^1$, $\mathbb S^3\times\mathbb S^3$ and $\mathbb S^7\times\mathbb S^7$
for $\DA=\RR,\CC,\HH,\OO$ respectively.
\end{lem}

\begin{proof}
$\Sigma$ is contained in the complement of a pole, which is contractible. A fiber bundle
over a contractible paracompact base is trivial, hence so is its restriction to $\Sigma$.
For $n=1$, $\Sigma=\{(0,\pm1)\}$ and, in the parametrisation of~\eqref{firsthopf},
$\pi^{-1}(\Sigma)=\{\theta\in\tfrac{\pi}{4}+\tfrac{\pi}{2}\ZZ\}$.
\end{proof}

\begin{rem}
\label{rem:fibersgroups}
For $n=1,2,4$ the fiber $\mathbb S^{n-1}$ is the group of unit elements of $\RR,\CC,\HH$,
namely $\Spin(1)\cong\ZZ_2$, $\Spin(2)\cong\mathrm U(1)$, $\Spin(3)\cong\mathrm{SU}(2)$,
so that $\pi^{-1}(\Sigma)\cong\Spin(m)\times\Spin(m)$, $m=1,2,3$. Only for $n=4$ is this
product again a spin group, $\Spin(3)\times\Spin(3)\cong\Spin(4)$, whereas
$\SO(4)\cong(\mathbb S^3\times\mathbb S^3)/\ZZ_2$. For $n=1$ the statement is a mere count
($\Sigma$ is disconnected), and for $n=8$ the factor $\mathbb S^7$ is not a group.
\end{rem}
%

\section{The Clifford map and $\Spin(9)$}
\label{sec:Spin9}

This section collects classical material \cite{Harvey,LawsonMichelsohn,OrneaPartonPiccinniVuletescu,PartonAxioms},
with proofs adapted to the present matrix language. Everything here takes place in the
associative algebra $\End_\RR(\OO^2)\cong M_{16}(\RR)$, so no bracketing issue arises.

\begin{defi}[Clifford map]
\label{def:Phi}
Equip $\OO^2\cong\RR^{16}$ with the inner product
$\bigl\langle\binom uv,\binom{u'}{v'}\bigr\rangle=\langle u,u'\rangle+\langle v,v'\rangle$.
For $X=X(y,p)\in\Herm_2(\OO)_0$ define the real-linear operator
\[
\Phi(X)=\begin{pmatrix} y\,\bu_8 & \Lm p\\ \Lm{\bar p} & -y\,\bu_8\end{pmatrix},\qquad
\Phi(X)\binom uv=\binom{yu+pv}{\bar pu-yv}.
\]
\end{defi}

\begin{prop}[Clifford relations]
\label{prop:clifford}
For all $X,X'\in\Herm_2(\OO)_0$:
\begin{enumerate}
\item $\Phi(X)$ is symmetric, and $\Phi(X)\Phi(X')+\Phi(X')\Phi(X)=2\Bf(X,X')\,\bu_{16}$;
in particular $\Phi(X)^2=Q(X)\bu_{16}$;
\item (reflection formula) if $Q(Y)=1$, then $\Phi(Y)\Phi(X)\Phi(Y)=\Phi\bigl(2\Bf(X,Y)Y-X\bigr)=-\Phi(r_Y X)$,
where $r_Y$ is the orthogonal reflection of $\Herm_2(\OO)_0$ in the hyperplane $Y^\perp$.
\end{enumerate}
\end{prop}

\begin{proof}
(1) Symmetry follows from Lemma~\ref{lem:Lx}(1). The diagonal blocks of the anticommutator
are $2yy'\bu_8+\Lm p\Lm{\bar p'}+\Lm{p'}\Lm{\bar p}=2(yy'+\langle p,p'\rangle)\bu_8$ by
Lemma~\ref{lem:Lx}(2), and the off-diagonal blocks are $y\Lm{p'}-y'\Lm p+y'\Lm p-y\Lm{p'}=0$.
(2) From (1), $\Phi(X)\Phi(Y)=2\Bf(X,Y)-\Phi(Y)\Phi(X)$; multiply on the left by $\Phi(Y)$
and use $\Phi(Y)^2=\bu_{16}$.
\end{proof}

By Proposition~\ref{prop:clifford}(1), $\Phi(X)^2=Q(X)\bu_{16}$ for all $X$. By the universal
property of Clifford algebras, $\Phi$ therefore extends to an algebra homomorphism
$C\ell(\RR^9,Q)\to\End_\RR(\OO^2)\cong M_{16}(\RR)$; from now on we work in this
\emph{associative} matrix algebra. For $Q(Y)=1$, $\Phi(Y)$ is symmetric with $\Phi(Y)^2=\bu_{16}$,
hence orthogonal; being traceless, it has eigenvalues $\pm1$ with multiplicity $8$ each, so
$\Phi(Y)\in\SO(16)$. Recall that $\Spin(9)\subset C\ell(\RR^9,Q)$ is the group generated by the
products $YY'$ of two unit vectors. Since $C\ell(\RR^9,Q)\cong M_{16}(\RR)\oplus M_{16}(\RR)$ and
$\Phi$ maps the even subalgebra isomorphically onto $M_{16}(\RR)$, $\Phi$ is injective on
$\Spin(9)$, and we identify
\begin{equation}
\label{eq:Spin9}
\Spin(9)=\bigl\langle\,\Phi(Y)\Phi(Y')\ :\ Y,Y'\in\mathbb S^8\,\bigr\rangle\subset\SO(16),
\end{equation}
acting on $\RR^{16}$ by its real spin representation \cite{LawsonMichelsohn,Harvey}.
Applying Proposition~\ref{prop:clifford}(2) twice, and using $\Phi(Y)^{-1}=\Phi(Y)$,
\[
\bigl(\Phi(Y)\Phi(Y')\bigr)\,\Phi(X)\,\bigl(\Phi(Y)\Phi(Y')\bigr)^{-1}
=\Phi(Y)\bigl(-\Phi(r_{Y'}X)\bigr)\Phi(Y)=\Phi(r_Yr_{Y'}X),
\]
so each generator acts on $\Herm_2(\OO)_0$ by the rotation $r_Yr_{Y'}$. Hence, for every
$S\in\Spin(9)$, $S\Phi(X)S^{-1}=\Phi(\rho(S)X)$ for a unique $\rho(S)\in\SO(9)$ (uniqueness
because $\Phi$ is injective), and $\rho$ is a homomorphism, onto by the Cartan--Dieudonn\'e
theorem. If $\rho(S)=\mathrm{id}$, then $S$ commutes with $\Phi(C\ell(\RR^9,Q))=M_{16}(\RR)$, so
$S=\pm\bu_{16}$; both occur, since $\Phi(Y)\Phi(-Y)=-\bu_{16}$. Thus $\rho:\Spin(9)\to\SO(9)$
is the double covering with kernel $\{\pm\bu_{16}\}$. Finally, $\SO(9)$ acts transitively on
$\mathbb S^8$ and the stabiliser of $X_N$ is $\SO(8)$, acting on $X_N^\perp\cong\RR^8$; its
preimage under $\rho$ is the subgroup $\Spin(8)$ generated by the $\Phi(Y)\Phi(Y')$ with
$Y,Y'\perp X_N$, whence $\mathbb S^8\cong\SO(9)/\SO(8)\cong\Spin(9)/\Spin(8)$.

\begin{prop}[$\Spin(9)$ equivariance of the octonionic Hopf map]
\label{thm:equivariance}
For $\xi=(u,v)\in\mathbb S^{15}$:
\begin{enumerate}
\item $\pi(\xi)$ is the unique $X\in\mathbb S^8$ such that $\Phi(X)\xi=\xi$;
\item for every $S\in\Spin(9)$, $\pi(S\xi)=\rho(S)\,\pi(\xi)$.
\end{enumerate}
\end{prop}

\begin{proof}
(1) With $y=|u|^2-|v|^2$, $p=2u\bar v$, alternativity gives $(u\bar v)v=|v|^2u$ and
$(v\bar u)u=|u|^2v$, so $\Phi(\pi(\xi))\xi=\bigl((|u|^2+|v|^2)u,(|u|^2+|v|^2)v\bigr)=\xi$.
Next, $\langle u,pv\rangle=\Re\bigl(u(\bar v\bar p)\bigr)=\Re\bigl((u\bar v)\bar p\bigr)=\langle u\bar v,p\rangle$
by~\eqref{eq:real-assoc}, and likewise $\langle v,\bar pu\rangle=\langle u\bar v,p\rangle$, so
$\langle\xi,\Phi(X)\xi\rangle=\Bf(X,\pi(\xi))$ for every $X\in\Herm_2(\OO)_0$. If $X\in\mathbb S^8$
and $\Phi(X)\xi=\xi$, then $1=\Bf(X,\pi(\xi))$, which forces $X=\pi(\xi)$ by Cauchy--Schwarz.
(2) $\Phi(\rho(S)\pi(\xi))\,S\xi=S\Phi(\pi(\xi))S^{-1}S\xi=S\xi$, and $S\xi\in\mathbb S^{15}$
since $S\in\SO(16)$; conclude by (1).
\end{proof}

\begin{rem}
Since $\Phi(X)$, $X\in\mathbb S^8$, is a traceless symmetric involution, (1) identifies the fiber
$\pi^{-1}(X)$ with the unit sphere of its $8$-dimensional $(+1)$-eigenspace, recovering the
fibers $\mathbb S^7$ of Proposition~\ref{prop:hopfmatrix} intrinsically.
\end{rem}

Proposition~\ref{thm:equivariance} is statement (S) of the Introduction. It is phrased entirely
in terms of $\Phi$ and of $\Spin(9)\subset\SO(16)$, and makes no use of octonionic
$2\times2$ matrix products. The question addressed in the rest of the paper is to what
extent the octonionic matrix expression $(gX)g^\dagger$ reproduces $\rho$.

\section{The RLLR/LRRL Bracketing Rule}
\label{sec:mnemonic}

\subsection{The bracketing problem}
\label{sec:problem}

Let $g=\begin{pmatrix}a&b\\c&d\end{pmatrix}\in\cL_\OO$ and
$X=X(y,p)\in\Herm_2(\OO)_0$. The product $Y:=gX$ involves only binary products:
\[
Y=\begin{pmatrix}ay+b\bar p & ap-by\\cy+d\bar p & cp-dy\end{pmatrix}.
\]
Each entry of $Yg^\dagger$ contains two terms of the form $y\,(xz)$, which are
unambiguous, and two genuine triple products $(\text{octonion})(\text{octonion})(\text{octonion})$,
each admitting two bracketings.

\begin{defi}[Bracketing symbols]
\label{def:brackets}
For $x,y,z\in\OO$ we write
$\Rp(x,y,z):=(xy)z$ and $\Lp(x,y,z):=x(yz)$.
\end{defi}

The eight triple products occurring in $(gX)g^\dagger$ are
\begin{equation}
\label{eq:triples}
\begin{array}{ll}
(1,1):\ (b,\bar p,\bar a),\ (a,p,\bar b), & (1,2):\ (b,\bar p,\bar c),\ (a,p,\bar d),\\
(2,1):\ (d,\bar p,\bar a),\ (c,p,\bar b), & (2,2):\ (d,\bar p,\bar c),\ (c,p,\bar d),
\end{array}
\end{equation}
so there are $2^8$ bracketing assignments a priori.

\subsection{The RLLR/LRRL rule}
\label{sec:rule}

\begin{defi}[RLLR/LRRL interlaced evaluation]
\label{def:RLLR}
Given $g=\begin{psmallmatrix}a&b\\c&d\end{psmallmatrix}\in\cL_\OO$ and
$X=X(y,p)\in\Herm_2(\OO)_0$, the \emph{interlaced evaluation}
$X'=\mathrm{Eval}(g;X)$ is defined as follows.

\noindent\textbf{Upper row} (pattern $\Rp\,\Lp\,\Lp\,\Rp$):
\begin{equation}
\label{eq:upper}
\begin{aligned}
X'_{11} &= y\bigl(|a|^2-|b|^2\bigr)
         + \underbrace{(b\bar p)\bar a}_{\Rp} + \underbrace{a(p\bar b)}_{\Lp}, \\
X'_{12} &= y(a\bar c - b\bar d)
         + \underbrace{b(\bar p\,\bar c)}_{\Lp} + \underbrace{(ap)\bar d}_{\Rp}.
\end{aligned}
\end{equation}

\noindent\textbf{Lower row} (pattern $\Lp\,\Rp\,\Rp\,\Lp$):
\begin{equation}
\label{eq:lower}
\begin{aligned}
X'_{21} &= y(c\bar a - d\bar b)
         + \underbrace{d(\bar p\,\bar a)}_{\Lp} + \underbrace{(cp)\bar b}_{\Rp}, \\
X'_{22} &= y\bigl(|c|^2-|d|^2\bigr)
         + \underbrace{(d\bar p)\bar c}_{\Rp} + \underbrace{c(p\bar d)}_{\Lp}.
\end{aligned}
\end{equation}
\end{defi}

The label ``RLLR/LRRL'' refers to the sequence of bracketing types read from left to right
across each row. In the associative cases all bracketings coincide, and the rule reduces to
the standard adjoint action; its content is exclusively octonionic.

\subsection{The conjugation mirror and Hermiticity}
\label{sec:mirror}

The RLLR/LRRL pattern is not an arbitrary convention: it is compatible with the
conjugation-mirror symmetry required to preserve Hermiticity, and admits the complementary
LRRL/RLLR prescription (Section~\ref{sec:admissible}). The extent to which these
requirements characterise the admissible bracketings is determined in
Proposition~\ref{prop:admissible} below.

\begin{lem}[Conjugation mirror]
\label{lem:mirror}
By~\eqref{eq:conj-swap}, conjugation maps each triple in~\eqref{eq:triples} to another
one, swapping $\Rp\leftrightarrow\Lp$:
\[
\overline{(b\bar p)\bar a}=a(p\bar b),\quad
\overline{b(\bar p\bar c)}=(cp)\bar b,\quad
\overline{(ap)\bar d}=d(\bar p\bar a),\quad
\overline{(d\bar p)\bar c}=c(p\bar d),
\]
together with the same identities with $\Rp$ and $\Lp$ exchanged. In particular, with the
RLLR/LRRL prescription, $X'_{21}=\overline{X'_{12}}$: this is the Hermitian off-diagonal
relation.
\end{lem}

\begin{prop}[Trace-free Hermitian output]
\label{prop:Hermoutput}
Let $\DA$ be a trace-associative algebra with standard involution (for instance $\OO$ or the
split octonions), $g\in\cL_\DA$ and $X\in\Herm_2(\DA)_0$. Then
$X'=\mathrm{Eval}(g;X)\in\Herm_2(\DA)_0$, i.e.\
\[
X'_{11}\in\RR,\qquad X'_{21}=\overline{X'_{12}},\qquad X'_{22}=-X'_{11},
\]
and
\begin{equation}
\label{eq:yprime}
X'_{11}=y\bigl(N(a)-N(b)\bigr)+2\Re\bigl(b\bar p\bar a\bigr),
\end{equation}
where $\Re(b\bar p\bar a)$ is independent of the bracketing by~\eqref{eq:real-assoc}.
\end{prop}

\begin{proof}
By Lemma~\ref{lem:mirror}, $a(p\bar b)=\overline{(b\bar p)\bar a}$, so
$(b\bar p)\bar a+a(p\bar b)=2\Re((b\bar p)\bar a)\in\RR$; similarly $X'_{22}\in\RR$, and
$X'_{21}=\overline{X'_{12}}$ by Lemma~\ref{lem:mirror}. For the trace,
$N(c)-N(d)=-(N(a)-N(b))$ by Lemma~\ref{lem:equalnorms}, and, by~\eqref{eq:real-cyclic},
$\Re(d\bar p\bar c)+\Re(b\bar p\bar a)=\Re\bigl(\bar p(\bar c d+\bar a b)\bigr)=0$
by~\eqref{eq:constraints2}.
\end{proof}

\subsection{Admissible bracketings, the complementary rule, and the order $g(Xg^\dagger)$}
\label{sec:admissible}

\begin{defi}[Admissible assignment]
\label{def:admissible}
A bracketing assignment for the eight triples~\eqref{eq:triples} is \emph{admissible} if,
for every $g\in\cL_\OO$ and $X\in\Herm_2(\OO)_0$, the resulting matrix lies in
$\Herm_2(\OO)_0$.
\end{defi}

\begin{lem}[Associator identities]
\label{lem:associds}
In $\OO$, $[\bar z,\bar y,\bar x]=[x,y,z]$ for all $x,y,z$. In particular, for the
triples~\eqref{eq:triples},
\[
[b,\bar p,\bar a]=[a,p,\bar b],\quad [d,\bar p,\bar c]=[c,p,\bar d],\quad
[c,p,\bar b]=[b,\bar p,\bar c],\quad [d,\bar p,\bar a]=[a,p,\bar d].
\]
\end{lem}

\begin{proof}
By the proof of Lemma~\ref{lem:realpart}, $\overline{[x,y,z]}=-[\bar z,\bar y,\bar x]$ and
$\overline{[x,y,z]}=-[x,y,z]$.
\end{proof}
We shall use the following explicit element of $\cL_\OO$ as a witness, written in the basis
$1,\ii,\mathrm j,\mathrm k,\ell,\ell\ii,\ell\mathrm j,\ell\mathrm k$ of Section~\ref{sec:algebras}
(recall $\ell q=(0,q)$ and $q\ell=\ell\bar q$ for $q\in\HH$):
\begin{equation}
\label{eq:hwitness}
h=\begin{pmatrix}\frac{1+\ii}{2} & \frac{\ell}{\sqrt2}\\[2pt] \frac{\ii}{\sqrt2} & \frac{-\ell+\ell\ii}{2}\end{pmatrix},
\qquad X_{\mathrm k}:=X(0,\mathrm k).
\end{equation}
One checks directly from~\eqref{eq:CD} that $h$ satisfies~\eqref{eq:constraints1}--\eqref{eq:constraints2}.
For $g=h$ and $p=\mathrm k$, the four associators of Lemma~\ref{lem:associds} are
\begin{equation}
\label{eq:hassoc}
[a,p,\bar b]=-\tfrac{1}{\sqrt2}\,\ell\mathrm j,\qquad
[c,p,\bar d]=\tfrac{1}{\sqrt2}(\ell\mathrm j+\ell\mathrm k),\qquad
[b,\bar p,\bar c]=-\ell\mathrm j,\qquad
[a,p,\bar d]=\tfrac12(\ell\mathrm j+\ell\mathrm k),
\end{equation}
and the last two are linearly independent over $\RR$.

\begin{prop}[The admissible evaluations]
\label{prop:admissible}
An assignment is admissible if and only if
\begin{enumerate}
\item[(a)] in each diagonal entry, the two triples carry opposite bracketings
($\Rp\Lp$ or $\Lp\Rp$), and
\item[(b)] the $(2,1)$ entry carries the conjugation mirror of the $(1,2)$ entry.
\end{enumerate}
Admissible assignments realise exactly four maps
$\mathrm{Eval}_{\epsilon_1\epsilon_2}$, indexed by the bracketings
$(\epsilon_1,\epsilon_2)\in\{\Rp,\Lp\}^2$ of the two triples $(b,\bar p,\bar c)$,
$(a,p,\bar d)$ in the $(1,2)$ entry. All four have the same diagonal~\eqref{eq:yprime},
and their $(1,2)$ entries differ by associators:
\[
\mathrm{Eval}_{\epsilon_1\epsilon_2}(g;X)_{12}-\mathrm{Eval}_{\epsilon'_1\epsilon'_2}(g;X)_{12}
\in\bigl\{0,\ \pm[b,\bar p,\bar c],\ \pm[a,p,\bar d],\ \pm[b,\bar p,\bar c]\pm[a,p,\bar d]\bigr\}.
\]
The RLLR/LRRL rule is $\mathrm{Eval}=\mathrm{Eval}_{\Lp\Rp}$.
\end{prop}

\begin{proof}
\emph{Sufficiency.} If (a) holds, the diagonal entries are real by the argument of
Proposition~\ref{prop:Hermoutput} and, being twice a bracketing-independent real part, do not
depend on which of the two opposite orders is chosen. If (b) holds, $X'_{21}=\overline{X'_{12}}$
by Lemma~\ref{lem:mirror}. The trace-free condition is then automatic
(Proposition~\ref{prop:Hermoutput}).

\emph{Necessity of (a).} Since $\Rp(x,y,z)=\Lp(x,y,z)+[x,y,z]$, if both triples of the
$(1,1)$ entry carry $\Rp$ (resp.\ $\Lp$), the entry equals
$y(N(a)-N(b))+2\Re(b\bar p\bar a)\pm[a,p,\bar b]$, using Lemma~\ref{lem:associds}; its imaginary
part is $\pm[a,p,\bar b]$. At $(g,X)=(h,X_{\mathrm k})$ this is $\mp\ell\mathrm j/\sqrt2\neq0$
by~\eqref{eq:hassoc}, so the entry is not real. The same argument applies to the $(2,2)$
entry, with imaginary part $\pm[c,p,\bar d]\neq0$ at $(h,X_{\mathrm k})$.

\emph{Necessity of (b).} Write the triple products of the $(1,2)$ entry as
$\beta_1(b,\bar p,\bar c)$ and $\beta_2(a,p,\bar d)$, and those of the $(2,1)$ entry as
$\gamma_1(c,p,\bar b)$ and $\gamma_2(d,\bar p,\bar a)$, with $\beta_i,\gamma_i\in\{\Rp,\Lp\}$, and
let $\hat\beta_i$ denote the bracketing opposite to $\beta_i$. By Lemma~\ref{lem:mirror},
$\overline{\beta_1(b,\bar p,\bar c)}=\hat\beta_1(c,p,\bar b)$ and
$\overline{\beta_2(a,p,\bar d)}=\hat\beta_2(d,\bar p,\bar a)$. Hence, by
Lemma~\ref{lem:associds},
\[
X'_{21}-\overline{X'_{12}}=s_1\,[b,\bar p,\bar c]+s_2\,[a,p,\bar d],\qquad s_i\in\{-1,0,1\},
\]
with $s_i=0$ if and only if $\gamma_i=\hat\beta_i$. At $(h,X_{\mathrm k})$ the two associators are
linearly independent by~\eqref{eq:hassoc}, so $X'_{21}=\overline{X'_{12}}$ forces
$s_1=s_2=0$, which is (b).

\emph{Exactly four maps.} Under (a) and (b) the map depends only on $(\epsilon_1,\epsilon_2)$, and
\[
\mathrm{Eval}_{\epsilon_1\epsilon_2}(g;X)_{12}-\mathrm{Eval}_{\epsilon'_1\epsilon'_2}(g;X)_{12}
=t_1[b,\bar p,\bar c]+t_2[a,p,\bar d],
\]
with $t_i\neq0$ exactly when $\epsilon_i\neq\epsilon'_i$. At $(h,X_{\mathrm k})$ this is nonzero
whenever $(\epsilon_1,\epsilon_2)\neq(\epsilon'_1,\epsilon'_2)$. Explicitly, the $(1,2)$ entries of
$\mathrm{Eval}_{\Lp\Rp},\mathrm{Eval}_{\Rp\Lp},\mathrm{Eval}_{\Lp\Lp},\mathrm{Eval}_{\Rp\Rp}$ at
$(h,X_{\mathrm k})$ are
\begin{equation}
\label{eq:hfour}
\ell\mathrm j,\qquad -\tfrac12(\ell\mathrm j+\ell\mathrm k),\qquad \tfrac12(\ell\mathrm j-\ell\mathrm k),\qquad 0,
\end{equation}
with vanishing diagonal in all four cases.
\end{proof}

\begin{rem}[What singles out RLLR/LRRL]
\label{rem:whyRLLR}
By Proposition~\ref{prop:admissible}, Hermiticity fixes the diagonal completely, but
leaves the $(1,2)$ entry free, giving four admissible maps. The RLLR/LRRL rule is the one in which the bracketing
types \emph{alternate} along each row ($\Rp\Lp\,|\,\Lp\Rp$), the second half of each row
being the mirror of the first; this is not a consequence of Hermiticity. The four rules are, however, distinguished
metrically: by Theorem~\ref{thm:blocks} below, each $\mathrm{Eval}_{\epsilon}$ is controlled by a
non-commutative $2\times2$ determinant $\Delta_\epsilon(g)$, and RLLR/LRRL is the rule
governed by the column determinant $ad-cb$; see Remark~\ref{rem:coldet}. The values
in~\eqref{eq:hfour} already show the effect: on $X_{\mathrm k}$, which has $Q=1$, the four rules
return points with $Q=1,\tfrac12,\tfrac12,0$. Note also that the
four maps coincide whenever $[b,\bar p,\bar c]=[a,p,\bar d]=0$, which is the case for the
transvections of Theorem~\ref{thm:transvection}; for the special form~\eqref{eq:gspecial},
$[a,p,\bar d]=[a,p,a]=0$ by flexibility, so only two distinct maps remain,
$\mathrm{Eval}_{\Lp\Rp}=\mathrm{Eval}_{\Lp\Lp}$ and $\mathrm{Eval}_{\Rp\Lp}=\mathrm{Eval}_{\Rp\Rp}$.
\end{rem}

\paragraph{The complementary rule and the order $g(Xg^\dagger)$.}
Exchanging $\Rp\leftrightarrow\Lp$ throughout gives the complementary rule LRRL/RLLR,
i.e.\ $E_2:=\mathrm{Eval}_{\Rp\Lp}$, while $E_1:=\mathrm{Eval}_{\Lp\Rp}$ is the RLLR/LRRL
rule. Both are admissible, they agree on the diagonal, and
\[
(E_1)_{12}-(E_2)_{12}=[a,p,\bar d]-[b,\bar p,\bar c],\qquad
(E_1)_{21}-(E_2)_{21}=\overline{[a,p,\bar d]-[b,\bar p,\bar c]}.
\]
Starting instead from the other global order $g(Xg^\dagger)$, one obtains the same eight
triples~\eqref{eq:triples}, now all naturally $\Lp$-bracketed. Applying the positional
pattern $\Rp\Lp\Lp\Rp/\Lp\Rp\Rp\Lp$ to the terms in the order in which they appear in the
expansion of $g(Xg^\dagger)$ reproduces $E_2$, and the complementary pattern reproduces
$E_1$. The two \emph{uniform} completions, all-$\Rp$ (the literal $(gX)g^\dagger$) and
all-$\Lp$ (the literal $g(Xg^\dagger)$), violate condition (a) and are not admissible.

\section{Transvections, quaternionic reduction, and the structure of the evaluation}
\label{sec:scope}

This section contains the main results of the paper (Section~\ref{sec:blocks}), preceded by
the compatible case, which is known in substance \cite{ManogueSchray1993,ManogueDray1999,DrayManogue2015} and which we treat in
the present bi-unitary and Clifford setting (Section~\ref{sec:transvections}). The bracketing rules of
Section~\ref{sec:mnemonic} admit a useful operator-theoretic interpretation: although
composition of real-linear operators on $\OO^2$ is intrinsically associative, its relationship
with octonionic matrix evaluation requires particular care. We first establish the
correspondence at the north pole, then on the transvections, where it is exact and
generates $\Spin(9)$ (Theorem~\ref{thm:transvection}). We then determine exactly what the
evaluation does for an arbitrary $g\in\cL_\OO$ (Theorem~\ref{thm:blocks}), which shows why the
correspondence cannot be extended beyond nested evaluations of transvections.

\begin{defi}[The operator $T_g$]
\label{def:Tg}
For $g=\begin{psmallmatrix}a&b\\c&d\end{psmallmatrix}\in M_2(\OO)$ define the real-linear map
\[
T_g=\begin{pmatrix}\Lm a&\Lm b\\ \Lm c&\Lm d\end{pmatrix}:\ \OO^2\to\OO^2,\qquad
T_g\binom{u}{v}=\binom{au+bv}{cu+dv}.
\]
By Lemma~\ref{lem:Lx}(1), $T_{g^\dagger}=T_g^{\mathrm t}=T_g^*$ is the adjoint of $T_g$.
\end{defi}

\begin{rem}
\label{rem:Tnotorth}
The condition $gg^\dagger=g^\dagger g=\buTwo$ does \emph{not} imply $T_g\in\mathrm O(16)$.
Indeed $T_gT_g^*=T_gT_{g^\dagger}$ has $(1,1)$ block
$\Lm a\Lm{\bar a}+\Lm b\Lm{\bar b}=\bu_8$ but $(1,2)$ block $\Lm a\Lm{\bar c}+\Lm b\Lm{\bar d}$,
which is the left multiplication by $a\bar c+b\bar d=0$ only if associators vanish. An explicit
counterexample is given in Proposition~\ref{prop:counter}(b). Bi-unitarity is a statement about
octonionic entries; orthogonality of $T_g$ is a statement about compositions of left
multiplications.
\end{rem}

\begin{rem}[Reading the operator at the unit]
\label{rem:readunit}
In $\End_\RR(\OO^2)$ one has, by block multiplication, $T_g\Phi(X)T_g^*=
\begin{psmallmatrix}\mathcal A&\mathcal B\\ \mathcal B^{\mathrm t}&\mathcal D\end{psmallmatrix}$ with,
for instance,
\[
\mathcal A=y\bigl(|a|^2-|b|^2\bigr)\bu_8+\Lm a\Lm p\Lm{\bar b}+\Lm b\Lm{\bar p}\Lm{\bar a},
\]
and $\Phi(X)$ as in Definition~\ref{def:Phi}. A composition such as $\Lm a\Lm p\Lm{\bar b}$
is in general not a left multiplication: $a\bigl(p(\bar bu)\bigr)$ differs from
$\bigl(a(p\bar b)\bigr)u$ by $-[a,p\bar b,u]-a[p,\bar b,u]$. Evaluating the blocks at $1$
yields uniformly $\Lp$-bracketed expressions; for the diagonal,
$\mathcal A(1)=y(|a|^2-|b|^2)+a(p\bar b)+b(\bar p\bar a)$, whose imaginary part
$-[a,p,\bar b]$ is in general nonzero. Hence $T_g\Phi(X)T_g^*\notin\Phi(\Herm_2(\OO)_0)$ for
general $g\in\cL_\OO$, and no bracketing rule can be read off from the operator composition:
the admissible rules of Proposition~\ref{prop:admissible}, RLLR/LRRL included, are
conventions on the octonionic side. The two sides agree exactly when $T_g$ normalises
$\Phi(\Herm_2(\OO)_0)$, as for the transvections below.
\end{rem}

\subsection{The north pole}
\label{sec:NP}

\begin{prop}[North pole]
\label{prop:mneqop}
For every $g\in\cL_\OO$, all $2^8$ bracketing assignments give the same matrix
$\mathrm{Eval}(g;X_N)=X(\pi(a,c))$ of Proposition~\ref{prop:hopfmatrix}. If moreover
$T_g\in\Spin(9)$, then $\mathrm{Eval}(g;X_N)=\rho(T_g)X_N$ and
$T_g\Phi(X_N)T_g^*=\Phi\bigl(\mathrm{Eval}(g;X_N)\bigr)$.
\end{prop}

\begin{proof}
At $X=X_N$ every triple product contains $p=0$ or $\bar p=0$, so all bracketings give zero.
If $T_g\in\Spin(9)$, then $\pi(T_g\xi_N)=\rho(T_g)\pi(\xi_N)$ with $\xi_N=(1,0)$, by
Proposition~\ref{thm:equivariance}; here $\pi(\xi_N)=X_N$ and $T_g\xi_N=(a,c)$.
\end{proof}

\subsection{The compatible case: transvections}
\label{sec:transvections}

\begin{defi}[Transvections]
\label{def:transvection}
For $\omega\in\RR$ and a unit octonion $\mu$, set
\[
g_{\omega,\mu}=\begin{pmatrix}\cos\omega & -\sin\omega\,\mu\\ \sin\omega\,\bar\mu & \cos\omega\end{pmatrix},
\qquad \cT=\{g_{\omega,\mu}\}.
\]
\end{defi}

\begin{rem}[Why ``transvection'']
\label{rem:terminology}
The name refers to the theory of Riemannian symmetric spaces, in which a \emph{transvection} is
a product $s_{\gamma(t)}s_{\gamma(0)}$ of two geodesic symmetries; it translates along the
geodesic $\gamma$ and rotates the plane of $\gamma$ \cite{Helgason1978}. On $\mathbb S^8$ the
geodesic symmetry at $Y$ is $s_Y(X)=2\Bf(X,Y)Y-X=-r_Y X$, so $s_Ys_{X_N}=r_Yr_{X_N}$; by
Theorem~\ref{thm:transvection} below, $\rho(T_{g_{\omega,\mu}})=s_{Y_{\omega,\mu}}s_{X_N}$ is exactly
the transvection of $\mathbb S^8$ along the great circle leaving $X_N$ in the direction $\mu$, by
arc length $2\omega$. This is unrelated to the linear-algebra notion of transvection
(elementary shear $x\mapsto x+\varphi(x)v$ with $\varphi(v)=0$): the maps considered here are
rotations, not shears. We prefer the term to ``elementary rotation'' because it records both the
geometry (translation along a geodesic issuing from the base point) and the factorisation into
two reflections that drives the proofs. We stress that $\cT$ is a set of matrices, not a group:
already in the quaternionic case a product of two transvections of $\mathbb S^4$ is in general not
a transvection, and over $\OO$ the matrix product does not even induce the composition of the
corresponding maps (Proposition~\ref{prop:counter}(c)). It is the operators
$T_{g}$, $g\in\cT$, that generate $\Spin(9)$.
\end{rem}

\begin{theorem}[Exactness on transvections]
\label{thm:transvection}
Let $g=g_{\omega,\mu}$ and $Y_{\omega,\mu}=X(\cos\omega,\sin\omega\,\mu)\in\mathbb S^8$.
\begin{enumerate}
\item $g\in\cL_\OO$ and $T_g=\Phi(Y_{\omega,\mu})\Phi(X_N)\in\Spin(9)$; in particular $T_g\in\SO(16)$.
\item All bracketings of $(gX)g^\dagger$ coincide, for every $X\in\Herm_2(\OO)_0$.
\item $T_g\Phi(X)T_g^*=\Phi\bigl(\mathrm{Eval}(g;X)\bigr)$ and $\mathrm{Eval}(g;X)=\rho(T_g)X$,
explicitly
\begin{equation}
\label{eq:transvrot}
y'=y\cos2\omega-\sin2\omega\,\langle p,\mu\rangle,\qquad
p'=p+\bigl(y\sin2\omega-2\sin^2\!\omega\,\langle p,\mu\rangle\bigr)\mu,
\end{equation}
which is the rotation by the angle $2\omega$ in the plane spanned by $X_N$ and $X(0,\mu)$,
fixing its orthogonal complement. In particular $\mathrm{Eval}(g;\cdot)\in\SO(9)$ preserves $Q$, and
$\mathrm{Eval}(g;X_N)=X(\cos2\omega,\sin2\omega\,\mu)$.
\item $\pi(T_g\xi)=\mathrm{Eval}(g;\pi(\xi))$ for all $\xi\in\mathbb S^{15}$.
\item The operators $T_g$, $g\in\cT$, generate $\Spin(9)$, and $\cT$ acts transitively on $\mathbb S^8$.
\end{enumerate}
\end{theorem}

\begin{proof}
Write $\kappa=\cos\omega$, $\sigma=\sin\omega$.
(1) The constraints~\eqref{eq:constraints1}--\eqref{eq:constraints2} reduce to $\kappa^2+\sigma^2=1$
and $\kappa\sigma\mu-\sigma\mu\kappa=0$. A direct block computation gives
$\Phi(Y_{\omega,\mu})\Phi(X_N)=\begin{psmallmatrix}\kappa&\sigma\Lm\mu\\ \sigma\Lm{\bar\mu}&-\kappa\end{psmallmatrix}
\begin{psmallmatrix}\bu&0\\0&-\bu\end{psmallmatrix}=
\begin{psmallmatrix}\kappa&-\sigma\Lm\mu\\ \sigma\Lm{\bar\mu}&\kappa\end{psmallmatrix}=T_g$.
(2) Since $a=d=\kappa$ is real, every triple containing $a$, $\bar a$, $d$ or $\bar d$ is
unambiguous. The only remaining triple is $(b,\bar p,\bar c)$ and its mirror $(c,p,\bar b)$:
with $b=-\sigma\mu$, $\bar c=\sigma\mu$, one gets $[b,\bar p,\bar c]=-\sigma^2[\mu,\bar p,\mu]=0$
by flexibility, and similarly for the mirror.
(3) By (1) and Proposition~\ref{prop:clifford}(2),
$T_g\Phi(X)T_g^*=\Phi(Y)\Phi(X_N)\Phi(X)\Phi(X_N)\Phi(Y)=\Phi(r_Yr_{X_N}X)=\Phi(\rho(T_g)X)$
with $Y=Y_{\omega,\mu}$. On the other hand, evaluating~\eqref{eq:upper} with (2) and
Lemma~\ref{lem:Lx}(3), $\mu\bar p\mu=2\langle\mu,p\rangle\mu-p$:
\[
X'_{11}=y(\kappa^2-\sigma^2)-\kappa\sigma(\mu\bar p+p\bar\mu)=y\cos2\omega-\sin2\omega\langle p,\mu\rangle,
\]
\[
X'_{12}=2y\kappa\sigma\mu-\sigma^2\mu\bar p\mu+\kappa^2p
= y\sin2\omega\,\mu-2\sigma^2\langle p,\mu\rangle\mu+p,
\]
which is~\eqref{eq:transvrot}. Computing $r_Yr_{X_N}X=r_Y X(-y,p)=2\Bf(X(y,-p),Y)Y-X(y,-p)$
with $\Bf(X(y,-p),Y)=\kappa y-\sigma\langle p,\mu\rangle$ gives the same expression. In the
coordinates $(y,\langle p,\mu\rangle)$ this is the rotation matrix of angle $2\omega$, and
the component of $p$ orthogonal to $\mu$ is unchanged.
(4) Proposition~\ref{thm:equivariance}(2) and (3).
(5) Since $Y_{\omega,\mu}$ runs over all of $\mathbb S^8$ and
$\Phi(Y)\Phi(Y')=\bigl(\Phi(Y)\Phi(X_N)\bigr)\bigl(\Phi(Y')\Phi(X_N)\bigr)^{-1}$, the $T_g$
generate the group~\eqref{eq:Spin9}. Transitivity follows from
$\mathrm{Eval}(g_{\omega,\mu};X_N)=X(\cos2\omega,\sin2\omega\,\mu)$.
\end{proof}

\begin{coro}[The $\Spin(9)$ action as composed evaluations]
\label{coro:words}
Every $S\in\Spin(9)$ is a product $S=T_{g_1}\cdots T_{g_k}$ with $g_i\in\cT$, and for any
such factorisation
\[
\rho(S)=\mathrm{Eval}(g_1;\cdot)\circ\cdots\circ\mathrm{Eval}(g_k;\cdot),
\qquad
\pi(S\xi)=\mathrm{Eval}\Bigl(g_1;\mathrm{Eval}\bigl(g_2;\cdots\mathrm{Eval}(g_k;\pi(\xi))\cdots\bigr)\Bigr).
\]
The right-hand side is independent of the chosen factorisation.
\end{coro}

\begin{proof}
$\rho$ is a homomorphism, and $\rho(T_{g_i})=\mathrm{Eval}(g_i;\cdot)$ by Theorem~\ref{thm:transvection}(3).
\end{proof}

\begin{rem}[Relation to the work of Manogue, Schray and Dray]
\label{rem:MSD}
The entries of $g_{\omega,\mu}$ lie in the complex subalgebra spanned by $1$ and $\mu$, so
$g_{\omega,\mu}$ is \emph{compatible} in the sense of \cite{ManogueSchray1993,ManogueDray1999}, and
part (2) of Theorem~\ref{thm:transvection} is an instance of their observation that
$X\mapsto MXM^\dagger$ is unambiguous for such $M$. Likewise, Corollary~\ref{coro:words} is the compact
counterpart of their result that nested compatible transformations generate the Lorentz group
$\SO(9,1)$ of $\Herm_2(\OO)$; in their construction the nesting depth can be bounded by two,
the nesting being needed only for transverse rotations that permute imaginary units, which they
generate by ``flips'' \cite{ManogueDray1999,DrayManogue2015}. The contribution of
Theorem~\ref{thm:transvection} lies in the identification $T_{g_{\omega,\mu}}=\Phi(Y_{\omega,\mu})\Phi(X_N)$,
which makes the link with the Clifford description of $\Spin(9)$ (Section~\ref{sec:Spin9})
immediate, and in the interpretation of $\rho(T_{g_{\omega,\mu}})$ as a transvection of
$\mathbb S^8$ (Remark~\ref{rem:terminology}).
\end{rem}

Thus statement (S) of the Introduction is realised by the RLLR/LRRL rule in the following
precise sense: the $\Spin(9)$ action is obtained by \emph{nesting} interlaced evaluations of
transvections, not by a single evaluation of an octonionic matrix product; see
Proposition~\ref{prop:counter}(c) and Theorem~\ref{thm:blocks}.

\begin{rem}[Diagonal elements]
\label{rem:diagonal}
For $g=\mathrm{diag}(u,v)$ with $|u|=|v|=1$, one has $g\in\cL_\OO$, $T_g\in\SO(16)$, and
$\mathrm{Eval}(g;X)=X\bigl(y,(up)\bar v\bigr)$, which preserves $Q$ and lies in $\SO(9)$.
However, the $(1,2)$ block of $T_g\Phi(X)T_g^*$ is $\Lm u\Lm p\Lm{\bar v}$, which is not, in
general, a left-multiplication operator. Consequently, although $T_g$ is orthogonal, it need not
belong to $\Spin(9)$ (it does, for instance, when $u=v=1$), and the Clifford-system interpretation
does not extend to arbitrary diagonal bi-unitary matrices. The $\Rp$ bracket in $(ap)\bar d$
matters here: the other admissible choice gives $u(p\bar v)$, which is also isometric but
different.
\end{rem}

\subsection{Main theorem: quaternionic reduction and the $5+4$ splitting}
\label{sec:blocks}

Away from the transvections, the behaviour of the evaluation is controlled by a simple
structural fact: every octonionic bi-unitary matrix is in fact quaternionic, for a suitable
quaternion subalgebra depending on the matrix.

\begin{lem}[Quaternionic reduction]
\label{lem:quatred}
For every $g=\begin{psmallmatrix}a&b\\c&d\end{psmallmatrix}\in\cL_\OO$ there is a quaternion
subalgebra $\HH'\subset\OO$ containing $a,b,c,d$. Consequently
\[
\cL_\OO=\bigcup_{\HH'\subset\OO}\mathrm{Sp}(2)_{\HH'},\qquad
\mathrm{Sp}(2)_{\HH'}:=\cL_{\HH'}\cong\mathrm{Sp}(2),
\]
the union running over the quaternion subalgebras of $\OO$. If $[a,b]:=ab-ba\neq0$, then
$\HH'$ is unique: it is the subalgebra generated by $a$ and $b$.
\end{lem}

\begin{proof}
If $a=0$, then $d=0$, $|b|=|c|=1$, and $b,c$ generate an associative subalgebra (Artin), of
dimension at most $4$; similarly if $b=0$. Let $a,b\neq0$, so $c,d\neq0$ and $|d|=|a|$. By
alternativity, $a\bar c=-b\bar d$ gives $\bar c=-a^{-1}(b\bar d)$, and $\bar cd=-\bar ab$ gives
$\bar c=-(\bar ab)d^{-1}$. Since $a^{-1}=\bar a/|a|^2$ and $d^{-1}=\bar d/|a|^2$, equating yields
$\bar a(b\bar d)=(\bar ab)\bar d$, i.e.\ $[\bar a,b,\bar d]=0$, i.e.\ $[a,b,d]=0$. Suppose first
$[a,b]\neq0$, and let $\HH'$ be the quaternion subalgebra generated by $a,b$. Choose a unit
$\ell'\perp\HH'$, so that $\OO=\HH'\oplus\ell'\HH'$ with the multiplication rule~\eqref{eq:CD}
(with $\ell'$ in place of $\ell$) \cite{SpringerVeldkamp,Baez}. For $x\in\HH'$ and $q\in\HH'$ this
rule gives
\begin{equation}
\label{eq:CDrel}
x(\ell'q)=\ell'(\bar xq),\qquad (\ell'q)x=\ell'(xq).
\end{equation}
Hence $(ab)(\ell'q)=\ell'(\bar b\bar aq)$ and $a(b(\ell'q))=\ell'(\bar a\bar bq)$, so
$[a,b,\ell'q]=\ell'\bigl([\bar b,\bar a]q\bigr)$, which vanishes only for $q=0$ because
$[\bar b,\bar a]=[b,a]\neq0$ and $\HH'$ is a division algebra. Thus $[a,b,d]=0$ forces $d\in\HH'$,
and then $c=-\overline{a^{-1}(b\bar d)}\in\HH'$. If $[a,b]=0$, then $a$ and $b$ lie in a common
subalgebra $\CC'$ of dimension at most $2$; let $\HH'$ be a quaternion subalgebra containing
$\CC'$ and $d$ (which exists since two elements generate an associative subalgebra of
dimension at most $4$); then $c\in\HH'$ by the same formula. The union formula follows, since
the constraints~\eqref{eq:constraints1}--\eqref{eq:constraints2} for entries in $\HH'$ are those
of $\mathrm{Sp}(2)$.
\end{proof}

\begin{defi}[The four determinants]
\label{def:Delta}
For $g=\begin{psmallmatrix}a&b\\c&d\end{psmallmatrix}$ put
\[
\Delta_{\Lp\Rp}(g)=ad-cb,\qquad \Delta_{\Rp\Lp}(g)=da-bc,\qquad
\Delta_{\Lp\Lp}(g)=da-cb,\qquad \Delta_{\Rp\Rp}(g)=ad-bc .
\]
\end{defi}

\begin{theorem}[$5+4$ splitting of the evaluation]
\label{thm:blocks}
Let $g\in\mathrm{Sp}(2)_{\HH'}\subset\cL_\OO$, choose a unit $\ell'\perp\HH'$, and decompose
$p\in\OO$ as $p=p_1+\ell'q$ with $p_1,q\in\HH'$. Then, for each admissible rule
$\epsilon\in\{\Lp\Rp,\Rp\Lp,\Lp\Lp,\Rp\Rp\}$,
\begin{equation}
\label{eq:blocks}
\mathrm{Eval}_\epsilon\bigl(g;X(y,p_1+\ell'q)\bigr)
= g\,X(y,p_1)\,g^\dagger\;+\;X\bigl(0,\ \ell'\,(\overline{\Delta_\epsilon(g)}\,q)\bigr),
\end{equation}
where $gX(y,p_1)g^\dagger$ is the ordinary (associative) quaternionic expression. Consequently:
\begin{enumerate}
\item $\mathrm{Eval}_\epsilon(g;\cdot)$ preserves the orthogonal decomposition
$\Herm_2(\OO)_0=(\RR\oplus\HH')\oplus\ell'\HH'\cong\RR^5\oplus\RR^4$; on $\RR\oplus\HH'$ it is the
standard action of $\mathrm{Sp}(2)_{\HH'}\cong\Spin(5)$, independent of $\epsilon$, and on
$\ell'\HH'$ it is $\ell'q\mapsto\ell'(\overline{\Delta_\epsilon(g)}\,q)$;
\item its singular values are $1$, with multiplicity $5$, and $|\Delta_\epsilon(g)|\in[0,1]$, with
multiplicity $4$; in particular $Q(\mathrm{Eval}_\epsilon(g;X))\le Q(X)$;
\item $\mathrm{Eval}_\epsilon(g;\cdot)\in\mathrm O(9)$ if and only if $|\Delta_\epsilon(g)|=1$, and then it
lies in $\SO(9)$.
\end{enumerate}
\end{theorem}

\begin{proof}
For $p\in\HH'$ all entries lie in the associative algebra $\HH'$, so every bracketing gives
$gX(y,p_1)g^\dagger$, the $\mathrm{Sp}(2)_{\HH'}$-action on $\RR\oplus\HH'$, which is isometric
(Proposition~\ref{prop:Qpres}(1)). By linearity it remains to treat $X=X(0,\ell'q)$, for which
$\bar p=-\ell'q$. Using~\eqref{eq:CDrel},
\[
(b\bar p)\bar c=-\ell'(\bar c\bar bq),\quad b(\bar p\bar c)=-\ell'(\bar b\bar cq),\quad
(ap)\bar d=\ell'(\bar d\bar aq),\quad a(p\bar d)=\ell'(\bar a\bar dq).
\]
For RLLR/LRRL ($\epsilon=\Lp\Rp$) this gives $X'_{12}=\ell'\bigl((\bar d\bar a-\bar b\bar c)q\bigr)
=\ell'\bigl(\overline{ad-cb}\,q\bigr)$, and the other three rules give the other three
determinants in the same way. All diagonal triples, e.g.\ $(b\bar p)\bar a=-\ell'(\bar a\bar bq)$,
lie in $\ell'\HH'$ and have zero real part, so $y'=0$. This proves~\eqref{eq:blocks} and (1).
For (2), left multiplication by $m\in\HH'$ on $\HH'$ is $|m|$ times an isometry, and
$|\Delta_\epsilon(g)|\le|a||d|+|b||c|=\cos^2\omega+\sin^2\omega=1$ by Lemma~\ref{lem:equalnorms}.
For (3), $\mathrm{Sp}(2)_{\HH'}$ is connected and left multiplication by a unit quaternion lies in
$\SO(4)$.
\end{proof}

\begin{coro}[Isometry criterion and special cases]
\label{coro:iso}
Let $g=\begin{psmallmatrix}a&b\\c&d\end{psmallmatrix}\in\cL_\OO$.
\begin{enumerate}
\item The RLLR/LRRL evaluation of $g$ is an isometry if and only if the column
determinant of $g$ has modulus one: $|ad-cb|=1$.
\item For transvections, $\Delta_\epsilon(g_{\omega,\mu})=1$ for all four rules; for diagonal
$g=\mathrm{diag}(u,v)$, $|\Delta_\epsilon(g)|=|u||v|=1$.
\item For the special form~\eqref{eq:gspecial},
\[
|\Delta_\epsilon(g)|^2=1-\bigl|[a,b]\bigr|^2\qquad\text{for all four rules,}
\]
so $\mathrm{Eval}_\epsilon(g;\cdot)$ is an isometry if and only if $ab=ba$, and it is singular, with
a $4$-dimensional kernel, if and only if $|[a,b]|=1$.
\end{enumerate}
\end{coro}

\begin{proof}
(1) and (2) follow from Theorem~\ref{thm:blocks} by direct substitution: for $g_{\omega,\mu}$,
$cb=bc=-\sin^2\omega$ and $a=d=\cos\omega$. For (3), write $a=|a|\alpha$, $b=|b|\beta$; the entries
lie in $\HH'$, which is associative, and $d=\bar a$, $c=-\bar\alpha\bar b\alpha$. For
$\epsilon=\Lp\Rp$, $\Delta=a\bar a-cb=|a|^2+|b|^2\,\bar\alpha\bar\beta\alpha\beta$, and
$\bar\alpha\bar\beta\alpha\beta=(\beta\alpha)^{-1}(\alpha\beta)$ has real part
$\langle\beta\alpha,\alpha\beta\rangle=1-\tfrac12|[\alpha,\beta]|^2$. Hence
$|\Delta|^2=(|a|^2+|b|^2)^2-|a|^2|b|^2|[\alpha,\beta]|^2=1-|[a,b]|^2$. The other three rules are
treated identically. Finally $|[a,b]|\le2|a||b|\le1$.
\end{proof}

\begin{rem}[An intrinsic reading of the RLLR/LRRL rule]
\label{rem:coldet}
The determinant attached to RLLR/LRRL, $\Delta_{\Lp\Rp}(g)=ad-cb$, is the \emph{column
determinant} of $g$, obtained by expanding along columns while keeping the column order
(first-column entries on the left); $\Delta_{\Rp\Rp}(g)=ad-bc$ is the corresponding row
determinant (see e.g.~\cite{Aslaksen1996} for such non-commutative determinants). The same
expression $ad-cb$ appears as the right pseudo-determinant $\Delta_R$ of the interlaced product
in~\cite{gazeau26}. Thus the four admissible rules, which Hermiticity cannot distinguish, are
distinguished metrically by four inequivalent $2\times2$ determinants over $\HH'$, and the
witness $h$ shows that they really differ: $\Delta_{\Lp\Rp}(h)=\ell\ii$ is unimodular, so the
RLLR/LRRL evaluation of $h$ is an isometry, whereas $|\Delta_{\Rp\Lp}(h)|=|\Delta_{\Lp\Lp}(h)|=\tfrac1{\sqrt2}$
and $\Delta_{\Rp\Rp}(h)=0$; compare~\eqref{eq:hfour}.
\end{rem}

\begin{prop}[Explicit examples]
\label{prop:counter}
In the convention~\eqref{eq:CD} with $\ell^2=-1$, let
\begin{equation}
\label{eq:g0}
g_0=\begin{pmatrix}\frac{1+\ii}{2} & \frac{\ell}{\sqrt2}\\[2pt] \frac{\ell\ii}{\sqrt2} & \frac{1-\ii}{2}\end{pmatrix},
\qquad
g_1=g_{\pi/4,\,\ii},\quad g_2=g_{\pi/4,\,\ell}\in\cT,\qquad
g_s=\frac1{\sqrt2}\begin{pmatrix}\ii&\mathrm j\\-\mathrm j&-\ii\end{pmatrix}.
\end{equation}
\begin{enumerate}
\item[(a)] $g_0$ is of the special form~\eqref{eq:gspecial}, with $a=\frac{1+\ii}2$,
$b=\frac{\ell}{\sqrt2}$, $|[a,b]|^2=\tfrac12$. Every admissible evaluation fails to preserve
$Q$: $E_1(g_0;X(0,\mathrm j))=X\bigl(0,\tfrac{\mathrm j-\mathrm k}{2}\bigr)$ and
$E_2(g_0;X(0,\mathrm j))=X\bigl(0,\tfrac{\mathrm j+\mathrm k}{2}\bigr)$, both with $Q=\tfrac12$.
\item[(b)] $T_{g_0}\notin\mathrm O(16)$: $T_{g_0}T_{g_0}^*\binom{0}{\mathrm j}
=\binom{(\ell\mathrm j+\ell\mathrm k)/\sqrt2}{\mathrm j}$, a vector of norm $\sqrt2$.
\item[(c)] $g_1,g_2$ lie in $\mathrm{Sp}(2)_{\HH'}$ with $\HH'$ spanned by $1,\ii,\ell,\ell\ii$, hence so
does $g_1g_2=\frac12\begin{psmallmatrix}1-\ell\ii & -\ii-\ell\\ -\ii-\ell & 1-\ell\ii\end{psmallmatrix}$.
But $\Delta_\epsilon(g_1g_2)=\tfrac12(1-\ell\ii)$ for all four rules, of modulus $\tfrac1{\sqrt2}$,
whereas $\Delta_\epsilon(g_1)=\Delta_\epsilon(g_2)=1$. Hence, for every admissible rule,
$\mathrm{Eval}(g_1g_2;\cdot)$ is not an isometry and differs from
$\mathrm{Eval}(g_1;\cdot)\circ\mathrm{Eval}(g_2;\cdot)$; for instance
$\mathrm{Eval}(g_1g_2;X(0,\mathrm j))=X\bigl(0,\tfrac{\mathrm j+\ell\mathrm k}{2}\bigr)$.
\item[(d)] $g_s$ is a quaternionic element of $\mathrm{Sp}(2)$ (the special form with $a=\ii/\sqrt2$,
$b=\mathrm j/\sqrt2$), and $\Delta_\epsilon(g_s)=0$ for all four rules: every admissible evaluation of
$g_s$ annihilates the $4$-dimensional subspace $X(0,\ell\HH)$, while acting on $\RR\oplus\HH$ by
the usual $\SO(5)$ rotation.
\end{enumerate}
\end{prop}

\begin{proof}
Finite computations with~\eqref{eq:CD}, combined with Theorem~\ref{thm:blocks} and
Corollary~\ref{coro:iso}. For instance, $c=-\bar\alpha\bar b\alpha$ with $\alpha=(1+\ii)/\sqrt2$
gives $c=\ell\ii/\sqrt2$ in (a), and $[a,b]=-\ell\ii/\sqrt2$. In (c), since
$\Delta(g_1g_2)\neq\Delta(g_1)\Delta(g_2)$, Theorem~\ref{thm:blocks} shows that the two maps differ on
$\ell'\HH'$, where they act by left multiplication by $\overline{\Delta(g_1g_2)}$ and by
$\overline{\Delta(g_1)}\,\overline{\Delta(g_2)}=1$ respectively, while they agree on $\RR\oplus\HH'$. All
values were also checked numerically, implementing~\eqref{eq:CD} with real $8\times8$
left-multiplication matrices; the script is available as supplementary material.
\end{proof}

\begin{rem}[Scope of the operator-theoretic picture]
\label{rem:operatorscope}
Composing two operators $T_g\circ T_h$ produces compositions $\Lm x\Lm y$, which are not of the
form $\Lm k$ for a single $k\in\OO$ unless $x$ and $y$ associate with everything; this is
the same failure of $\{\Lm x\}_{x\in\OO}$ to close under composition that underlies
octonionic triality \cite{Baez}. It is the common root of Remark~\ref{rem:Tnotorth} and of
Remark~\ref{rem:diagonal}. On the side of the evaluation, Theorem~\ref{thm:blocks} identifies
the obstruction precisely: the set
$\cL_\OO^{\mathrm{iso}}=\{g\in\cL_\OO:\ \mathrm{Eval}(g;\cdot)\in\mathrm O(9)\}$ is the level set
$\{|ad-cb|=1\}$, and the failure of $\mathrm{Eval}(gh;\cdot)=\mathrm{Eval}(g;\cdot)\circ\mathrm{Eval}(h;\cdot)$
for $g,h$ in a common $\mathrm{Sp}(2)_{\HH'}$ is exactly the failure of the column determinant to be
multiplicative on $\mathrm{Sp}(2)$ (the Dieudonn\'e determinant, which is multiplicative, is trivial
on $\mathrm{Sp}(2)$; the column determinant is not).
\end{rem}

\section{What extends to general algebras}
\label{sec:general}

\begin{defi}[Interlaced evaluation for general $\DA$]
\label{def:generaleval}
For an algebra with standard involution $\DA$, $g\in\cL_\DA$ and $X\in\Herm_2(\DA)_0$, define
$\mathrm{Eval}_\DA(g;X)$ by the RLLR/LRRL prescription of Definition~\ref{def:RLLR}, replacing
$\OO$ by $\DA$.
\end{defi}

We separate three statements of different strength.

\begin{theorem}[Hermiticity: general]
\label{thm:general}
Let $\DA$ be a trace-associative algebra with standard involution. For all $g\in\cL_\DA$ and
$X\in\Herm_2(\DA)_0$, $\mathrm{Eval}_\DA(g;X)\in\Herm_2(\DA)_0$, with diagonal
entry~\eqref{eq:yprime}. This applies in particular to $\RR,\CC,\HH,\OO$ and to the split
octonions.
\end{theorem}

\begin{proof}
Proposition~\ref{prop:Hermoutput}, whose proof uses only~\eqref{eq:real-assoc},
\eqref{eq:real-cyclic}, \eqref{eq:conj-swap} and the constraints~\eqref{eq:constraints1}--\eqref{eq:constraints2}.
\end{proof}

\begin{prop}[Preservation of $Q$: under explicit hypotheses]
\label{prop:Qpres}
$Q\bigl(\mathrm{Eval}_\DA(g;X)\bigr)=Q(X)$ holds in each of the following cases:
\begin{enumerate}
\item $\DA$ is associative and $g\in\cL_\DA$ (standard: $gXg^\dagger$ with $g$ unitary, and
$Q(X)=-\det X$ in the commutative cases, $Q(X)\buTwo=X^2$ in general);
\item $\DA=\OO$ and $g\in\cT$ (Theorem~\ref{thm:transvection});
\item $\DA=\OO$ and $g$ diagonal (Remark~\ref{rem:diagonal});
\item more generally, $\DA=\OO$ and $|ad-cb|=1$, which is also necessary
(Theorem~\ref{thm:blocks}).
\end{enumerate}
It fails for some $g\in\cL_\OO$ of the special form (Proposition~\ref{prop:counter}(a)). For the
split octonions $Q$ is indefinite and the analogue of (2) holds with the same proof for
transvections built from $\mu$ with $N(\mu)=1$; we do not pursue this here.
\end{prop}

\begin{proof}
In the associative case $(gXg^\dagger)^2=gX^2g^\dagger=Q(X)gg^\dagger=Q(X)\buTwo$. The other cases
were proved above.
\end{proof}

\begin{prop}[Action property: requires additional structure]
\label{prop:action}
If $\DA$ is associative, $g\mapsto\mathrm{Eval}_\DA(g;\cdot)$ is a group action of $\cL_\DA$ on
$\mathbb S^{\dim\DA}$. For $\DA=\OO$, $\mathrm{Eval}(gh;\cdot)\neq\mathrm{Eval}(g;\cdot)\circ\mathrm{Eval}(h;\cdot)$
in general, even for $g,h\in\cT$, because the column determinant is not multiplicative
(Proposition~\ref{prop:counter}(c)); a genuine action of
$\Spin(9)$ is obtained by composing evaluations along words in $\cT$
(Corollary~\ref{coro:words}).
\end{prop}

\begin{rem}[Split octonions and beyond]
For the split octonions ($\ell^2=+1$ in~\eqref{eq:CD}), the norm $N$ is nondegenerate of
signature $(4,4)$, so they are not covered by Definition~\ref{def:normed}, but they are covered
by Definition~\ref{def:invol}; Theorem~\ref{thm:general} applies. The $2\times2$ mnemonic for
split octonionic multiplication is studied in \cite{gazeau26}. For the higher Cayley--Dickson
algebras (sedenions, \dots), conjugation is still an anti-automorphism, so the mirror
Lemma~\ref{lem:mirror} holds; alternativity fails, so Lemma~\ref{lem:Lx}, the Clifford
relations and everything in Section~\ref{sec:scope} would have to be reconsidered.
\end{rem}

\section{Conclusion}
\label{sec:conclusion}

We have developed a common $2\times2$ matrix framework for the division-algebra Hopf
fibrations and determined how far it extends to the octonions.

\paragraph{Main results.}
Every octonionic bi-unitary matrix is quaternionic, for a quaternion subalgebra $\HH'$ depending
on it, so that $\cL_\OO$ is a union of copies of $\mathrm{Sp}(2)$ (Lemma~\ref{lem:quatred}).
Requiring Hermitian output leaves exactly four bracketing rules (Proposition~\ref{prop:admissible}),
and each of them splits orthogonally as the $\Spin(5)$ rotation of $\RR\oplus\HH'$ plus left
multiplication by a non-commutative determinant on $\HH'^\perp$ (Theorem~\ref{thm:blocks}). For
RLLR/LRRL this determinant is the column determinant $ad-cb$: the evaluation is an isometry
exactly when $|ad-cb|=1$, a contraction otherwise, and singular when $ad-cb=0$. This answers, for
all four admissible rules, the question of where octonionic $2\times2$ matrix calculus preserves
the geometry of $\mathbb S^8$ outside the compatible class used in earlier work
\cite{ManogueSchray1993,ManogueDray1999,DrayManogue2015}, and it gives the RLLR/LRRL rule an intrinsic characterisation which
Hermiticity alone does not provide.

\paragraph{The compatible case.}
For the octonionic transvections $g_{\omega,\mu}$, which are compatible matrices, all
bracketings of $(gX)g^\dagger$ coincide and the resulting map is the $\Spin(9)$
rotation~\eqref{eq:intro-rot}; since the transvections generate $\Spin(9)$, nested evaluations
realise the whole $\Spin(9)$ action and the equivariance of the octonionic Hopf map
(Theorem~\ref{thm:transvection}, Corollary~\ref{coro:words}). This is the compact counterpart of the
nested Lorentz transformations of Manogue, Schray and Dray \cite{ManogueSchray1993,ManogueDray1999,DrayManogue2015}; our
contribution here is the Clifford factorisation of the generators and their interpretation as
transvections of $\mathbb S^8$.

\paragraph{Convention-free and general statements.}
The Hopf map is the image of the north pole, $(gX_N)g^\dagger=X(\pi(a,c))$ with $(a,c)$ the first
column of $g$, for every $g\in\cL_\DA$ and every $\DA\in\{\RR,\CC,\HH,\OO\}$, with no bracketing
convention (Proposition~\ref{prop:hopfmatrix}). The Hermiticity of the admissible evaluations
holds over any trace-associative algebra with standard involution, including the split octonions
(Theorem~\ref{thm:general}).

\paragraph{The central message.}
Non-associativity does not prevent a $2\times2$ octonionic description of the full $\Spin(9)$
symmetry; it prevents that description from being realised through unrestricted multiplication
of octonionic matrices. The symmetry survives through compositions of suitably chosen
elementary transformations, the transvections, for which every bracketing agrees, as in the
nested constructions of \cite{ManogueSchray1993,ManogueDray1999,DrayManogue2015}. Beyond them,
the obstruction is not a matter of bracketing convention: even within a single copy
$\mathrm{Sp}(2)_{\HH'}$, where the matrix product is associative, the evaluation of a product
differs from the composition of evaluations because the column determinant is not
multiplicative (Proposition~\ref{prop:counter}(c)). For the physical uses of
$X\mapsto gXg^\ddagger$ recalled in the Introduction, this means that the octonionic version of
the formalism is faithful only as a calculus of composed elementary transformations.

\paragraph{Open questions.}
(i) Give a geometric interpretation of the four determinants $\Delta_\epsilon$ on
$\cL_\OO=\bigcup_{\HH'}\mathrm{Sp}(2)_{\HH'}$, for instance in terms of the Grassmannian of
quaternion subalgebras and of triality, and describe the level sets $\{|\Delta_\epsilon|=1\}$.
(ii) Clarify the relation between the column determinant governing RLLR/LRRL and the
pseudo-determinants of the interlaced product of \cite{gazeau26}, and whether the latter product
restricted to $\cL^{\mathrm{iso}}_\OO$ is compatible with the evaluation. (iii) Relate the nested
evaluations of Corollary~\ref{coro:words} to a non-associative product on (a subset of)
$\cL_\OO$, possibly through the triality maps.



\begin{thebibliography}{99}

\bibitem{Hopf}
H.~Hopf,
\"{U}ber die Abbildungen der dreidimensionalen Sph\"{a}re auf die Kugelfl\"{a}che,
\emph{Math. Ann.} \textbf{104} (1931), 637--665.

\bibitem{Adams}
J.~F.~Adams,
On the non-existence of elements of Hopf invariant one,
\emph{Ann. of Math.} (2) \textbf{72} (1960), 20--104.

\bibitem{Husemoller}
D.~Husemoller,
\emph{Fibre Bundles}, 3rd ed.,
Graduate Texts in Mathematics 20, Springer, 1994.

\bibitem{PenroseRindler}
R.~Penrose and W.~Rindler, \emph{Spinors and Space-Time}, Vol.~1,
Cambridge University Press, 1984.

\bibitem{Sudbery1984}
A.~Sudbery,
Division algebras, (pseudo)orthogonal groups and spinors,
\emph{J. Phys. A: Math. Gen.} \textbf{17} (1984), 939--955.

\bibitem{KugoTownsend1983}
T.~Kugo and P.~K.~Townsend, Supersymmetry and the division algebras,
\emph{Nucl.\ Phys.\ B} \textbf{221} (1983), 357--380.

\bibitem{Evans1988}
J.~M.~Evans, Supersymmetric Yang--Mills theories and division algebras,
\emph{Nucl.\ Phys.\ B} \textbf{298} (1988), 92--108.

\bibitem{BaezHuerta2010}
J.~C.~Baez and J.~Huerta, Division algebras and supersymmetry I,
in \emph{Superstrings, Geometry, Topology, and $C^*$-algebras},
Proc.\ Sympos.\ Pure Math.\ \textbf{81}, AMS, 2010, 65--80.

\bibitem{Takahashi1963}
R.~Takahashi, Sur les repr\'esentations unitaires des groupes de Lorentz g\'en\'eralis\'es,
\emph{Bull.\ Soc.\ Math.\ France} \textbf{91} (1963), 289--433.

\bibitem{Ahlfors1985}
L.~V.~Ahlfors, M\"obius transformations and Clifford numbers, in
\emph{Differential Geometry and Complex Analysis}, Springer, 1985, 65--73.

\bibitem{gazpejtod26} J.-P. Gazeau, H. Pejhan, and I. Todorov, \emph{Massless Representations in Conformal Space and Their dS Restrictions}, forthcoming in Cambridge Monographs on Mathematical Physics, Cambridge University Press; arXiv:2601.18433 (2026).

\bibitem{Fock1935}
V.~Fock, Zur Theorie des Wasserstoffatoms, \emph{Z.\ Phys.} \textbf{98} (1935), 145--154.

\bibitem{BarutKleinert1967}
A.~O.~Barut and H.~Kleinert, Transition probabilities of the hydrogen atom from
noncompact dynamical groups, \emph{Phys.\ Rev.} \textbf{156} (1967), 1541--1545.

\bibitem{Gazeau1978}
J.-P.~Gazeau, Four Euclidean conformal group in atomic calculations: exact analytical
expressions for the bound-bound two-photon transition matrix elements  in the H atom,   
\emph{J.\ Math.\ Phys.} \textbf{19} (1978), 1041; erratum \emph{ibid.} \textbf{21} (1980).

\bibitem{Gazeau1979}
J.-P.~Gazeau, On the four Euclidean conformal group structure of the Sturmian operator,
\emph{Lett.\ Math.\ Phys.} \textbf{3} (1979), 285--292.

\bibitem{MosseriDandoloff2001}
R.~Mosseri and R.~Dandoloff, Geometry of entangled states, Bloch spheres and Hopf
fibrations, \emph{J.\ Phys.\ A} \textbf{34} (2001), 10243--10252.

\bibitem{BernevigChen2003}
B.~A.~Bernevig and H.-D.~Chen, Geometry of the three-qubit state, entanglement and
division algebras, \emph{J.\ Phys.\ A} \textbf{36} (2003), 8325--8339.

\bibitem{JordanvNW1934}
P.~Jordan, J.~von Neumann and E.~Wigner, On an algebraic generalization of the quantum
mechanical formalism, \emph{Ann.\ of Math.} \textbf{35} (1934), 29--64.

\bibitem{GunaydinGursey1973}
M.~G\"unaydin and F.~G\"ursey, Quark structure and octonions,
\emph{J.\ Math.\ Phys.} \textbf{14} (1973), 1651--1667.

\bibitem{ManogueSchray1993}
C.~A.~Manogue and J.~Schray,
Finite Lorentz transformations, automorphisms, and division algebras,
\emph{J. Math. Phys.} \textbf{34} (1993), 3746--3767.

\bibitem{ManogueDray1999}
C.~A.~Manogue and T.~Dray,
Octonionic M\"obius transformations,
\emph{Mod. Phys. Lett. A} \textbf{14} (1999), 1243--1255.

\bibitem{DrayManogue2015}
T.~Dray and C.~A.~Manogue,
\emph{The Geometry of the Octonions},
World Scientific, 2015.

\bibitem{gazeau24}
J.-P.~Gazeau, The Language of Spheres in Physics, \emph{Universe} \textbf{10}(3) (2024), 117.

\bibitem{Harvey}
F.~R.~Harvey,
\emph{Spinors and Calibrations},
Academic Press, 1990.

\bibitem{LawsonMichelsohn}
H.~B.~Lawson and M.-L.~Michelsohn,
\emph{Spin Geometry},
Princeton University Press, 1989.

\bibitem{OrneaPartonPiccinniVuletescu}
L.~Ornea, M.~Parton, P.~Piccinni and V.~Vuletescu,
Spin$(9)$ geometry of the octonionic Hopf fibration,
\emph{Transformation Groups} \textbf{18} (2013), 845--864.

\bibitem{PartonAxioms}
M.~Parton,
The Role of Spin$(9)$ in Octonionic Geometry,
\emph{Axioms} \textbf{7} (2018), 72.

\bibitem{Baez}
J.~C.~Baez,
The Octonions,
\emph{Bull.\ Amer.\ Math.\ Soc.} \textbf{39} (2002), 145--205.

\bibitem{Schafer}
R.~D.~Schafer,
\emph{An Introduction to Nonassociative Algebras},
Academic Press, 1966.

\bibitem{Okubo}
S.~Okubo,
\emph{Introduction to Octonion and Other Non-Associative Algebras in Physics},
Cambridge University Press, 1995.

\bibitem{SpringerVeldkamp}
T.~A.~Springer and F.~D.~Veldkamp,
\emph{Octonions, Jordan Algebras and Exceptional Groups},
Springer, 2000.

\bibitem{ConwaySmith}
J.~H.~Conway and D.~A.~Smith,
\emph{On Quaternions and Octonions},
A K Peters/CRC Press, 2003.

\bibitem{gazeau26}
J.-P.~Gazeau,
A Matrix Mnemonic for Octonions and Cayley--Dickson Tower,
\textit{Advances in Applied Clifford Algebras} \textbf{36}, 34 (2026),
\url{https://doi.org/10.1007/s00006-026-01452-8}.

\bibitem{Dieudonne1943}
J.~Dieudonn\'e,
Les d\'eterminants sur un corps non commutatif,
\emph{Bull.\ Soc.\ Math.\ France} \textbf{71} (1943), 27--45.

\bibitem{Dieudonne1955}
J.~Dieudonn\'e,
\emph{La G\'eom\'etrie des Groupes Classiques},
Ergebnisse der Mathematik und ihrer Grenzgebiete, Springer, 1955 (3rd ed.\ 1971).

\bibitem{Helgason1978}
S.~Helgason,
\emph{Differential Geometry, Lie Groups, and Symmetric Spaces},
Academic Press, 1978 (Graduate Studies in Mathematics reprint, AMS, 2001).

\bibitem{Aslaksen1996}
H.~Aslaksen,
Quaternionic determinants,
\emph{Math.\ Intelligencer} \textbf{18} (1996), no.~3, 57--65.

\end{thebibliography}
\end{document}